\documentclass[10pt, oneside]{article}
\usepackage[english]{babel}
\usepackage{amsmath}
\usepackage{amsthm}
\usepackage{amssymb}
\usepackage{graphicx}
\usepackage{pst-all}
\usepackage{authblk}
\usepackage{enumerate}
\usepackage{color,graphicx}
\usepackage{graphics}
\usepackage{xspace,colortbl}
\usepackage[latin1]{inputenc}
\usepackage{float}

\usepackage{fancyvrb}%%% una clase de verbtim
\DefineVerbatimEnvironment
{MiVerbatim}{Verbatim}{fontsize=\footnotesize,frame=single,label=\emph{Mathematica
8.0},framesep=2mm,numbers=left}

\title{An Introduction to Inversion in an Ellipse}
\author[1]{José L. Ramírez \thanks{josel.ramirez@ima.usergioarboleda.edu.co}}
\affil[1]{Departamento de Matemáticas,  Universidad Sergio Arboleda, Bogotá, Colombia}

\begin{document}
\newtheorem{theorem}{Theorem}
\newtheorem{definition}{Definition}
\newtheorem{corollary}{Corollary}
\newtheorem{example}{Example}
\newtheorem{lemma}{Lemma}
\newtheorem{proposition}{Proposition}
\maketitle
\setlength{\parindent}{0pt}

\begin{abstract}
In this paper we study the inversion in an ellipse and some properties, which generalizes the classical inversion with respect to a circle. We also study  the inversion in an ellipse of lines, ellipses and other curves. Finally, we generalize the Pappus Chain with respect to ellipses and the Pappus Chain Theorem.  \\
\textbf{Keywords:} Inversion, elliptic inversion, elliptic inversion of curves, Elliptic Pappus Chain.

\end{abstract}
\date
\maketitle

\section{Introduction}
In this paper we study the elliptic inversion, which was introduced in  \cite{CHI}, and some related properties to the distance of elliptic inverse points,  cross ratio, harmonic conjugates and the elliptic inversion of different curves. Elliptic inversion generalizes the classical inversion, which has a lot of properties and applications, see \cite{BLA, OGI, PED}.\\

The outline of this paper is as follow. In Section 2 we define the inversion respect to an ellipse.  In Section 3 we study some basic properties of the inversion  in an ellipse and its relations with the cross ratio and the   harmonic conjugates. We also study the cartesian coordinates of elliptic points. In Section 4 we describe the inversion  in an ellipse of lines and conics. Finally, in Section 5 we introduce the Elliptic Pappus Chain and we apply the inversion in an ellipse  to proof  the  generalize Pappus  Chain Theorem. 
\section{Elliptic Inversion}

\begin{definition}\label{dinve}
Let $E$ be an ellipse centered at a point  $O$ with focus $F_1$  and $F_2$ in $\mathbb{R}^2$. The inversion in the ellipse  $E$ or Elliptic Inversion respect to $E$  is the mapping
$\psi:\mathbb{R}^2\setminus \{O\} \longmapsto
\mathbb{R}^2\setminus \{O\}$ defined by $\psi(P)=P'$, where $P'$ lies on the ray   $\stackrel{\longrightarrow}{OP}$ and $OP\cdot OP'=(OQ)^2$, where $Q$ is the point of  intersection of the ray $\stackrel{\longrightarrow}{OP}$ and the ellipse $E$.
\end{definition}
The point $P'$ is said to be the \emph{elliptic inverse} of $P$ in the ellipse $E$, or with respect to the ellipse $E$,  $E$  is called the \emph{ellipse of inversion}, $O$ is called  the \emph{center of  inversion}, and the number $OQ=w$ is called the \emph{radius of inversion},  see Figure \ref{fig:elipsepartes}. The inversion with respect to the ellipse $E$, center of inversion $O$ and radius of  inversion $w>0$ is denoted by $\mathcal{E}(O,w)$.  Unlike the classical case, here the radius is not constant.

\begin{figure}[h]
    \begin{center}
\psset{xunit=1.0cm,yunit=1.0cm,algebraic=true,dotstyle=o,dotsize=3pt 0,linewidth=0.8pt,arrowsize=3pt 2,arrowinset=0.25}
\begin{pspicture*}(-2.8,-1.6)(4.2,2.16)
\rput{0}(0,0){\psellipse(0,0)(2.5,1.5)}
\rput[tl](-2.6,2.1){$E$}
\psline(0,0)(3.72,1.6)
\psline[linestyle=dashed,dash=2pt 2pt](-2.5,0)(2.5,0)
\begin{scriptsize}
\psdots[dotstyle=*,linecolor=blue](2,0)
\rput[bl](2,0.12){\blue{$F_2$}}
\psdots[dotstyle=*,linecolor=blue](-2,0)
\rput[bl](-2,0.12){\blue{$F_1$}}
\psdots[dotstyle=*,linecolor=blue](0,0)
\rput[bl](-0.3,0.12){\blue{$O$}}
\psdots[dotstyle=*,linecolor=blue](3.72,1.6)
\rput[bl](3.8,1.72){\blue{$P$}}
\psdots[dotstyle=*,linecolor=blue](1.11,0.48)
\rput[bl](0.7,0.58){\blue{$P'$}}
\psdots[dotstyle=*,linecolor=blue](2.03,0.87)
\rput[bl](1.9,1){\blue{$Q$}}
\end{scriptsize}
\end{pspicture*}
    \end{center}
    \caption{Inversion in an Ellipse.}
    \label{fig:elipsepartes}
\end{figure}
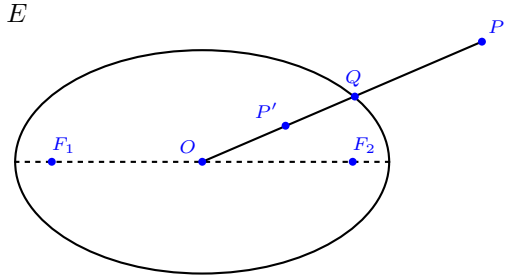

The elliptic inversion is an involutive mapping, i.e., $\psi\left(\psi\left(P\right)\right)=P$. The fixed points are the points on the ellipse $E$. Indeed,  if $F$ is a fixed point, $\psi(F)=F$, then $OF\cdot OF=(OF)^2=(OQ)^2$, then  $OF=OQ$ and as Q lies on the ray $\stackrel{\longrightarrow}{OF}$, then $F=Q$.

\begin{proposition}\label{intexteriorelipse}
If $P$ is in the exterior of $E$ then $P'$ is interior to $E$, and conversely.
\end{proposition}
\begin{proof}
Let $P$ be an exterior point of  $\mathcal{E}(O,w)$, then $w<OP$. If $P'$ is the elliptic inverse of $P$, then $OP\cdot OP'=w^2$. Hence $w^2=OP\cdot OP'>w \cdot OP'$ and  $OP'<w$.
\end{proof}

Inversion in an ellipse inversion does not hold for the center of inversion $O$, as in the usual definition. However, we can add to the Euclidean plane a single point at infinite $O_{\infty}$, which is the inverse of the center of any elliptic inversion. This plane is denoted by $\mathbb{R}^2_{\infty}$.  We now have a one-to-one map of our extended plane.
\begin{definition}\label{dinve2}
Let $E$ be an ellipse centered at a point  $O$ in $\mathbb{R}^2_{\infty}$, the elliptic inversion in this ellipse is the mapping
$\psi:\mathbb{R}^2_{\infty}\longmapsto \mathbb{R}^2_{\infty}$ defined by $\psi(P)=P'$, where $P'$ lies on the ray   $\stackrel{\longrightarrow}{OP}$ and $(OP)(OP')=(OQ)^2$, where $Q$ is the point of  intersection of the ray $\stackrel{\longrightarrow}{OP}$ and the ellipse $E$, $\psi(O_\infty)=O$ and $\psi(O)=O_\infty$.
\end{definition}

\section{Basic Properties}

\begin{theorem}\label{diseliptica}
Let $P$ and $T$ be different points. Let $P'$ and $T'$ their respective elliptic inverse points respect to $\mathcal{E}(O,w)$ and $\mathcal{E}(O,u)$. Then

\begin{enumerate}[i.]
    \item If $P$, $T$ and $O$ are not collinear, then  \[ P'T'=\frac{\sqrt{\left(w^2-u^2\right)\left(w^2(OT)^2-u^2(OP)^2\right)+w^2u^2(PT)^2}}{OP\cdot OT}. \]
    \item If $P$, $T$ and $O$ are collinear, then
    \[P'T'=\frac{w^2PT}{OP\cdot OT}.\]
\end{enumerate}
\end{theorem}

\begin{proof}
 \emph{i.} If $P, T$ and $O$ are not collinear. Then  $P',  T'$ and $O$ are not also collinear, see Figure \ref{fig:disinveliptica}.

\begin{figure}[h]
    \begin{center}
       \newrgbcolor{qqwuqq}{0 0.39 0}
\psset{xunit=0.85cm,yunit=0.85cm,algebraic=true,dotstyle=o,dotsize=3pt 0,linewidth=0.8pt,arrowsize=3pt 2,arrowinset=0.25}
\begin{pspicture*}(-3.86,-1.6)(2.86,2.46)
\rput{0}(0,0){\psellipse(0,0)(2.5,1.5)}
\rput[tl](2.24,2.12){$E$}
\psline[linestyle=dashed,dash=3pt 3pt](-2.5,0)(2.5,0)
\psline[linecolor=qqwuqq](0,0)(1.46,2.02)
\psline[linecolor=qqwuqq](0,0)(-3.73,2.03)
\psline[linecolor=qqwuqq](-3.73,2.03)(0.68,0.94)
\psline[linecolor=qqwuqq](-0.92,0.5)(1.46,2.02)
\rput[tl](-0.16,0.4){$\alpha$}
\begin{scriptsize}
\psdots[dotstyle=*,linecolor=blue](2,0)
\rput[bl](1.92,-0.36){\blue{$F_2$}}
\psdots[dotstyle=*,linecolor=blue](-2,0)
\rput[bl](-2.08,-0.34){\blue{$F_1$}}
\psdots[dotstyle=*,linecolor=blue](0,0)
\rput[bl](0.04,-0.24){\blue{$O$}}
\psdots[dotstyle=*,linecolor=blue](-1.85,1)
\rput[bl](-1.96,0.65){\blue{$Q$}}
\psdots[dotstyle=*,linecolor=blue](1,1.35)
\rput[bl](0.95,1.05){\blue{$S$}}
\psdots[dotstyle=*,linecolor=blue](-0.92,0.5)
\rput[bl](-0.98,0.16){\blue{$P$}}
\psdots[dotstyle=*,linecolor=red](-3.73,2.03)
\rput[bl](-3.66,2.14){\red{$P'$}}
\psdots[dotstyle=*,linecolor=blue](1.46,2.02)
\rput[bl](1.54,2.14){\blue{$T$}}
\psdots[dotstyle=*,linecolor=red](0.68,0.94)
\rput[bl](0.8,0.66){\red{$T'$}}
\end{scriptsize}
\end{pspicture*}
    \end{center}
    \caption{Distance and Inverse Points.}
    \label{fig:disinveliptica}
\end{figure}
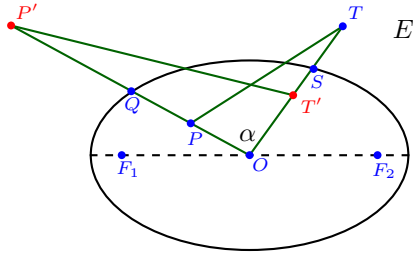

Let $\alpha$ be the measure of the  angle  $\angle P'OT'$, then  by law of cosines
\begin{align}\label{ecuaa15}
(P'T')^2=(OP')^2+(OT')^2-2 \cdot OP' \cdot OT' \cdot\cos \alpha
\end{align}

From $OP\cdot OP'=(OQ)^2=w^2$ and $OT\cdot OT'=(OS)^2=u^2$, we have  $OP'=\frac{w^2}{OP}$ and  $OT'=\frac{u^2}{OT}$, where $Q$ and $S$ are respectively the points of intersection of  rays $\stackrel{\longrightarrow}{OP}$ and  $\stackrel{\longrightarrow}{OT}$ with $E$, see Figure \ref{fig:disinveliptica}.

Replacing these values in (\ref{ecuaa15}):
\begin{align}\label{ecuaa16}
(P'T')^2=\frac{w^4}{(OP)^2}+\frac{u^4}{(OT)^2}-2\frac{w^2u^2}{OP\cdot OT}\cos \alpha
\end{align}

As $\alpha$ is also the measure of the  angle  $\angle POT$, then  by law of cosines
\begin{align*}
(PT)^2&=(OP)^2+(OT)^2-2\cdot OP\cdot OT\cdot \cos \alpha\\
2\cos \alpha&=\frac{(OP)^2+(OT)^2-(PT)^2}{OP\cdot OT}
\end{align*}
Replacing  in (\ref{ecuaa16}):
\begin{align*}
(P'T')^2&=\frac{w^4}{(OP)^2}+\frac{u^4}{(OT)^2}-\frac{w^2u^2}{OP\cdot OT}\left(\frac{(OP)^2+(OT)^2-(PT)^2}{OP\cdot OT}\right)\\
&=\frac{w^2(OT)^2\left(w^2-u^2\right)-u^2(OP)^2\left(w^2-u^2\right)+w^2u^2(PT)^2}{(OP)^2(OT)^2}\\
&=\frac{\left(w^2-u^2\right)\left(w^2(OT)^2-u^2(OP)^2\right)+w^2u^2(PT)^2}{(OP)^2(OT)^2}
\end{align*}
Hence
    \[P'T'=\frac{\sqrt{\left(w^2-u^2\right)\left(w^2(OT)^2-u^2(OP)^2\right)+w^2u^2(PT)^2}}{OP\cdot OT}
\]

\emph{ii.} When $P, Q$ are $O$ collinear, then  $OQ=w=u=OS$. Therefore
    \[P'T'=\frac{w^2\cdot PT}{OP\cdot OT} \qedhere \]
\end{proof}

Note that if $E$ is a circumference, then $OQ=w=u=OS$. Hence
\begin{align*}
P'T'&=\frac{\sqrt{\left(w^2-w^2\right)\left(w^2(OT)^2-w^2(OP)^2\right)+w^2w^2(PT)^2}}{(OP)(OT)}\\
&=\frac{\sqrt{w^4(PT)^2}}{OP\cdot OT}\\
&=\frac{w^2\cdot PT}{OP\cdot OT}
\end{align*}
where $w$ is the radius of the circumference.

\subsection{Inversion in an Ellipse and Cross Ratio}

Suppose that $A, B, C$ and $D$ are four distinct points on a line $l$; we define their \emph{cross ratio} $\left\{AB, CD\right\}$ by
\begin{align*}
\left\{AB, CD\right\}=\frac{\overrightarrow{AC}\cdot\overrightarrow{BD}}{\overrightarrow{AD}\cdot\overrightarrow{BC}}
\end{align*}
where $\overrightarrow{AB}$ denote the signed distance from $A$ to $B$. The cross ratio is an invariant under inversion in a circle whose center is not any of the four points $A, B, C$ or $D$, see \cite{BLA}. However, the inversion in an ellipse does not preserve the cross ratio, for example see Figure \ref{fig:cocientedobleinveliptica}.

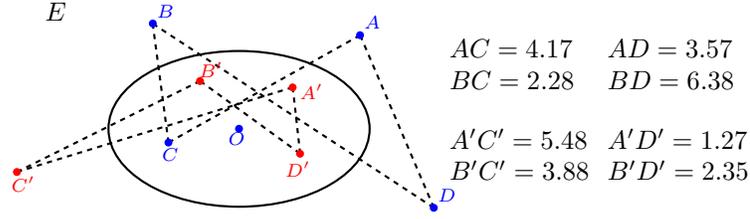
\begin{figure}[h]
    \begin{center}
\psset{xunit=0.7cm,yunit=0.7cm,algebraic=true,dotstyle=o,dotsize=3pt 0,linewidth=0.8pt,arrowsize=3pt 2,arrowinset=0.25}
\begin{pspicture*}(-4.58,-1.76)(10,2.42)
\rput{0}(0,0){\psellipse(0,0)(2.5,1.5)}
\rput[tl](-3.68,2.4){$E$}
\psline[linestyle=dashed,dash=2pt 2pt](2.3,1.78)(-1.34,-0.26)
\psline[linestyle=dashed,dash=2pt 2pt](-1.64,2)(3.7,-1.5)
\psline[linestyle=dashed,dash=2pt 2pt](2.3,1.78)(3.7,-1.5)
\psline[linestyle=dashed,dash=2pt 2pt](-1.64,2)(-1.34,-0.26)
\psline[linestyle=dashed,dash=2pt 2pt](1.02,0.79)(-4.22,-0.82)
\psline[linestyle=dashed,dash=2pt 2pt](-0.74,0.91)(1.16,-0.47)
\psline[linestyle=dashed,dash=2pt 2pt](1.02,0.79)(1.16,-0.47)
\psline[linestyle=dashed,dash=2pt 2pt](-0.74,0.91)(-4.22,-0.82)
\rput[tl](7,1.1){$BD=6.38$}
\rput[tl](4,1.7){$AC=4.17$}
\rput[tl](7,1.7){$AD=3.57$}
\rput[tl](4,1.1){$BC=2.28$}
\rput[tl](4,0){$A'C'=5.48$}
\rput[tl](7,0){$A'D'=1.27$}
\rput[tl](4,-0.6){$B'C'=3.88$}
\rput[tl](7,-0.6){$B'D'=2.35$}
\begin{scriptsize}
\psdots[dotstyle=*,linecolor=blue](0,0)
\rput[bl](-0.2,-0.3){\blue{$O$}}
\psdots[dotstyle=*,linecolor=blue](2.3,1.78)
\rput[bl](2.38,1.9){\blue{$A$}}
\psdots[dotstyle=*,linecolor=red](1.02,0.79)
\rput[bl](1.16,0.56){\red{$A'$}}
\psdots[dotstyle=*,linecolor=blue](-1.64,2)
\rput[bl](-1.56,2.12){\blue{$B$}}
\psdots[dotstyle=*,linecolor=blue](-1.34,-0.26)
\rput[bl](-1.46,-0.6){\blue{$C$}}
\psdots[dotstyle=*,linecolor=blue](3.7,-1.5)
\rput[bl](3.78,-1.38){\blue{$D$}}
\psdots[dotstyle=*,linecolor=red](-0.74,0.91)
\rput[bl](-0.76,0.98){\red{$B'$}}
\psdots[dotstyle=*,linecolor=red](-4.22,-0.82)
\rput[bl](-4.32,-1.2){\red{$C'$}}
\psdots[dotstyle=*,linecolor=red](1.16,-0.47)
\rput[bl](0.88,-0.9){\red{$D'$}}
\end{scriptsize}
\end{pspicture*}
    \end{center}
    \caption{Elliptic Inversion and Cross Ratio.}
    \label{fig:cocientedobleinveliptica}
\end{figure} 
\begin{align*}
\left\{AB,CD\right\}&=\frac{AC\cdot BD}{AD\cdot BC}=\frac{4.17\cdot 2.28}{3.57\cdot 2.28}\approx 1.168, \\ \left\{A'B',C'D'\right\}&=\frac{A'C'\cdot B'D'}{A'D'\cdot B'C'}=\frac{5.48\cdot 2.35)}{1.27\cdot 3.88}\approx 2.613.
\end{align*}

\subsection{Inversion in an Ellipse and Harmonic Conjugates}
If $A$ and $B$ are two points on a line $l$, any pair of points $P$ and $Q$ on $l$ for which
\begin{align*}
\frac{AP}{PB}=\frac{AQ}{QB},
\end{align*}
are said to divide $\overline{AB}$\emph{ harmonically}. The points $P$ and $Q$ are called \emph{harmonic conjugates with respect to $A$ and $B$}. It is clear that two distinct points $P$ and $Q$ are harmonic conjugates with respecto to  $A$ and $B$ if and only if $\left\{AB,PQ\right\}=1$.

\begin{theorem}\label{armoinveliptica}
Let $E$ be an ellipse with center $O$,  and $\overline{Q_{1}Q}_{2}$ a diameter of  $E$. Let $P$ and $P'$ be distinct points  of  the ray $\stackrel{\longrightarrow}{OQ}_{1}$, which divide the segment  $\overline{Q_{1}Q}_{2}$ internally and externally. Then $P$ and $P'$ are harmonic conjugates with respect to  $Q_{1}$ and $Q_{2}$ if and only if $P$ and $P'$ are elliptic inverse points with respect  $E$.
\end{theorem}

\begin{proof}
Suppose that $P$ and $P'$ are harmonic points with respect to  $Q_{1}$ and $Q_{2}$. Then
\begin{align*}
\left\{Q_{1}Q_{2},PP'\right\}&=1,\\
\frac{Q_{1}P\cdot Q_{2}P'}{Q_{1}P'\cdot Q_{2}P}&=1.
\end{align*}
Note that if $P$ divide the segment  $\overline{Q_{1}Q}_{2}$ internally  and $P\in\stackrel{\longrightarrow}{OQ_{1}}$. Then $Q_{1}P=OQ_{1}-OP=w-OP$ and $Q_{2}P=OQ_{2}+OP=w+OP$. Moreover, $P'$ divide the segment  $\overline{Q_{1}Q}_{2}$ externally and $P'\in \stackrel{\longrightarrow}{OQ_{1}}$. Then $Q_{1}P'=OP'-OQ_{1}=OP'-w$ and $Q_{2}P'=OQ_{2}+OP'=w+OP'$. Hence
 \begin{align*}
\frac{(w-OP)(k+OP')}{(OP'-w)(w+OP)}&=1,\\
(w-OP)(w+OP')&=(OP'-w)(k+OP).
\end{align*}
Simplifying this equation, we have  $OP\cdot OP'=w^2$. Therefore  $P$ and $P'$ are elliptic inverse points with respect to $E$.

Conversely, if $P$ and $P'$ are elliptic inverse points with respect to $\mathcal{E}(O,w)$, the proof is similar.
\end{proof}

\subsection{Inversion in an Ellipse and Cartesian Coordinates}

\begin{theorem}\label{coordenadaselipse}
Let $E$ be an ellipse with center $O$ and equation $\frac{x^2}{a^2}+\frac{y^2}{b^2}=1$,  $a$ and $b$ are respectively the semi-major axis and semi-minor axis. Let $P=(u,v)$  and $P'=(x, y)$ be a pair of elliptic points with respect to $E$. Then
\begin{align}\label{ecuainvelipse}
x&=\frac{a^2b^2u}{b^2u^2+a^2v^2},\\
y&=\frac{a^2b^2v}{b^2u^2+a^2v^2}.
\end{align}
\end{theorem}
\begin{proof}
Let $E$ be an ellipse with equation $\frac{x^2}{a^2}+\frac{y^2}{b^2}=1$. Suppose that $P=(u,v)$ is an exterior point to $E$. Let  $T=(x_{1},y_{1})$ and $M=(x_{2},y_{2})$ be the points of contact of the tangent lines to $E$ from $P$, see Figure \ref{fig:coordenadasinvelipse}. Then the tangent lines $\stackrel{\longleftrightarrow}{PT}$ and $\stackrel{\longleftrightarrow}{PM}$ have the following equations \cite[p. 186]{LEM}:
\begin{align}\label{coor1}
b^2x_{1}x+a^2y_{1}y&=a^2b^2,\\
b^2x_{2}x+a^2y_{2}y&=a^2b^2.
\end{align}
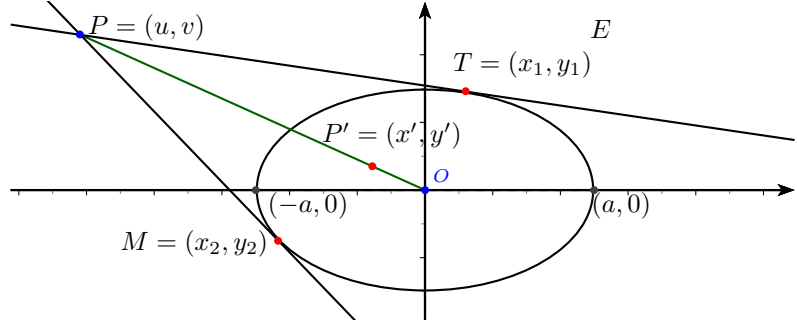
\begin{figure}[h]
    \begin{center}
\newrgbcolor{qqwuqq}{0 0.39 0}
\psset{xunit=0.9cm,yunit=0.9cm,algebraic=true,dotstyle=o,dotsize=3pt 0,linewidth=0.8pt,arrowsize=3pt 2,arrowinset=0.25}
\begin{pspicture*}(-6.12,-1.92)(5.46,2.78)
\psaxes[labelFontSize=\scriptstyle,xAxis=true,yAxis=true,labels=none,Dx=1,Dy=1,ticksize=-2pt 0,subticks=2]{->}(0,0)(-6.12,-1.92)(5.46,2.78)
\rput{0}(0,0){\psellipse(0,0)(2.5,1.5)}
\rput[tl](2.44,2.52){$E$}
\psline[linestyle=dashed,dash=3pt 3pt](-2.5,0)(2.5,0)
\psline[linecolor=qqwuqq](0,0)(-5.1,2.3)
\psplot{-6.12}{5.46}{(-8.8-3.05*x)/2.93}
\psplot{-6.12}{5.46}{(--8.8-0.84*x)/5.7}
\rput[tl](0.42,2.05){$T=(x_1,y_1)$}
\rput[tl](-4.5,-0.6){$M=(x_2,y_2)$}
\rput[tl](2.46,-0.04){$(a,0)$}
\rput[tl](-2.32,-0.04){$(-a,0)$}
\rput[tl](-4.98,2.6){$P=(u,v)$}
\rput[tl](-1.52,1){$P'=(x',y')$}
\begin{scriptsize}
\psdots[dotstyle=*,linecolor=blue](0,0)
\rput[bl](0.12,0.12){\blue{$O$}}
\psdots[dotstyle=*,linecolor=darkgray](-2.5,0)
\psdots[dotstyle=*,linecolor=darkgray](2.5,0)
\psdots[dotstyle=*,linecolor=blue](-5.1,2.3)
\psdots[dotstyle=*,linecolor=red](-0.78,0.35)
\psdots[dotstyle=*,linecolor=red](0.6,1.46)
\psdots[dotstyle=*,linecolor=red](-2.17,-0.75)
\end{scriptsize}
\end{pspicture*}
    \end{center}
    \caption{Inversion in an Ellipse and  Cartesian Coordinates.}
    \label{fig:coordenadasinvelipse}
\end{figure}

Particularly $P=(u,v)$ satisfies these equations. Hence
\begin{align}\label{coor2}
b^2x_{1}u+a^2y_{1}v&=a^2b^2,\\
b^2x_{2}u+a^2y_{2}v&=a^2b^2.\label{coor3}
\end{align}
Equating equations (\ref{coor2}) and (\ref{coor3})
\begin{align*}
b^2x_{1}u+a^2y_{1}v&=b^2x_{2}u+a^2y_{2}v,\\
-\frac{b^2u}{a^2v}&=\frac{y_{1}-y_{2}}{x_{1}-x_{2}}.
\end{align*}
Then the line $\stackrel{\longleftrightarrow}{TM}$ has slope $-\frac{b^2u}{a^2v}$. Therefore, $\stackrel{\longleftrightarrow}{TM}$ has the following equation
\begin{align}
y-y_{1}&=-\frac{b^2u}{a^2v}\left(x-x_{1}\right),\\
a^2vy-a^2vy_{1}&=-b^2ux+b^2ux_{1},\\
a^2vy+b^2ux&=b^2ux_{1}+a^2vy_{1}.\label{coor5}
\end{align}
Replacing (\ref{coor2}) in (\ref{coor5}), we have
\begin{align}\label{coor6}
a^2vy+b^2ux&=a^2b^2,
\end{align}
i.e., (\ref{coor6}) is the equation of the line $\stackrel{\longleftrightarrow}{TM}$.
On the other hand, the line  $\stackrel{\longleftrightarrow}{OP}$ has slope  $\frac{v}{u}$, then its equation is $y=\frac{v}{u}x$. As $P'$ is the meeting point of the lines  $\stackrel{\longleftrightarrow}{TM}$ and $\stackrel{\longleftrightarrow}{OP}$, then
\begin{align*}
a^2v\left(\frac{v}{u}x\right)+b^2ux&=a^2b^2\\
\left(a^2v^2+b^2u^2\right)x&=ua^2b^2\\
x&=\frac{ua^2b^2}{a^2v^2+b^2u^2}
\end{align*}
and
\begin{align*}
y&=\frac{va^2b^2}{a^2v^2+b^2u^2}
\end{align*}
When $P$  is an interior point of $E$, the proof is analogous.
\end{proof}

When $a=b=1$, i.e., when $E$ is a circle, we obtain
\begin{align*}
\psi:(u,v)&\longmapsto \left(\frac{u}{v^2+u^2},\frac{v}{v^2+u^2} \right)
\end{align*}

\section{Elliptic Inversion of Curves}

In this section we study  the inversion in an ellipse of lines, ellipses and other curves.  If a point $P$ moves on a curve $\mathcal{C}$, and $P'$, the elliptic inverse of $P$ with respect to the $E$ moves on a curve $\mathcal{C}'$, the curve $\mathcal{C}'$  is called the elliptic inverse of $\mathcal{C}$. It is evident that $\mathcal{C}$ is the elliptic inverse of $\mathcal{C}'$ in $E$.

\begin{theorem}\label{invrelipse}
\begin{enumerate}[i.]
\item The elliptic inverse of a line $l$ which pass through the center of the elliptic inversion is the line itself.
    \item  The elliptic inverse of a line $l$ which does not pass through the center of the elliptic inversion is an ellipse which pass through the center of inversion, see Figure \ref{fig:rectasinversionelipse2}.
\end{enumerate}
\end{theorem}
\begin{proof}
\textit{i.} Let $E$ be an ellipse of inversion with equation $\frac{x^2}{a^2}+\frac{y^2}{b^2}=1$ and  $l$ a line with  equation $Mx + Ny = 0$. Applying $\psi$ to $Mx + Ny = 0$ gives $Mx + Ny = 0$. Indeed
\begin{align*}
Mx + Ny &= 0\\
M\left(\frac{a^2b^2x}{b^2x^2+a^2y^2}\right)+N\left(\frac{a^2b^2y}{b^2x^2+a^2y^2}\right)&=0\\
Ma^2b^2x+Na^2b^2y&=0\\
Mx + Ny &= 0
\end{align*}

 \textit{ii.} Let $E$ be an ellipse of inversion with equation $\frac{x^2}{a^2}+\frac{y^2}{b^2}=1$ and  $l$ a line with  equation  $Mx + Ny + P = 0$  ($P\neq 0$). Applying $\psi$ to $Mx + Ny + P = 0$ gives $\frac{x^2}{a^2}+\frac{y^2}{b^2}+\frac{M}{P}x+\frac{N}{P}y=0$. Indeed
\begin{align*}
Mx + Ny + P &= 0\\
M\left(\frac{a^2b^2x}{b^2x^2+a^2y^2}\right)+N\left(\frac{a^2b^2y}{b^2x^2+a^2y^2}\right)+ P&=0\\
M\left(a^2b^2x\right)+N\left(a^2b^2y\right)+ \left(b^2x^2+a^2y^2\right)P&=0\\
\frac{x^2}{a^2}+\frac{y^2}{b^2}+\frac{M}{P}x+\frac{N}{P}y&=0
\end{align*}
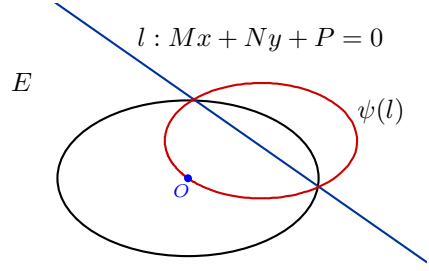
\begin{figure}[h]
	\begin{center}
\newrgbcolor{qqttzz}{0 0.2 0.6}
\newrgbcolor{ccqqqq}{0.8 0 0}
\psset{xunit=0.7cm,yunit=0.7cm,algebraic=true,dotstyle=o,dotsize=3pt 0,linewidth=0.8pt,arrowsize=3pt 2,arrowinset=0.25}
\begin{pspicture*}(-4.02,-1.68)(4.54,3.32)
\rput{0}(0,0){\psellipse(0,0)(2.5,1.5)}
\rput[tl](-3.36,2){$E$}
\psplot[linecolor=qqttzz]{-4.02}{4.54}{(--7.75-3.44*x)/4.92}
\pscustom[linecolor=ccqqqq]{\moveto(-0.36,1.04)
\lineto(-0.33,1.09)
\lineto(-0.29,1.15)
\lineto(-0.24,1.22)
\lineto(-0.17,1.29)
\lineto(-0.07,1.38)
\lineto(-0.01,1.42)
\lineto(0.05,1.46)
\lineto(0.09,1.49)
\lineto(0.1,1.49)
\lineto(0.11,1.5)
\lineto(0.11,1.5)
\lineto(0.11,1.5)
\lineto(0.11,1.5)
\lineto(0.11,1.5)
\lineto(0.11,1.5)
\lineto(0.11,1.5)
\moveto(2.49,-0.16)
\lineto(2.49,-0.16)
\lineto(2.49,-0.16)
\lineto(2.49,-0.16)
\lineto(2.49,-0.16)
\lineto(2.49,-0.16)
\lineto(2.49,-0.16)
\lineto(2.49,-0.16)
\lineto(2.49,-0.16)
\lineto(2.49,-0.16)
\lineto(2.49,-0.16)
\lineto(2.48,-0.16)
\lineto(2.48,-0.16)
\lineto(2.48,-0.16)
\lineto(2.48,-0.16)
\lineto(2.48,-0.17)
\lineto(2.47,-0.17)
\lineto(2.46,-0.17)
\lineto(2.43,-0.19)
\lineto(2.38,-0.21)
\lineto(2.32,-0.23)
\lineto(2.27,-0.25)
\lineto(2.21,-0.26)
\lineto(2.16,-0.28)
\lineto(2.11,-0.29)
\lineto(2.06,-0.31)
\lineto(1.96,-0.33)
\lineto(1.87,-0.34)
\lineto(1.78,-0.36)
\lineto(1.69,-0.37)
\lineto(1.61,-0.37)
\lineto(1.54,-0.38)
\lineto(1.46,-0.38)
\lineto(1.39,-0.38)
\lineto(1.33,-0.38)
\lineto(1.27,-0.38)
\lineto(1.21,-0.38)
\lineto(1.16,-0.37)
\lineto(1.1,-0.37)
\lineto(1.01,-0.36)
\lineto(0.95,-0.35)
\lineto(0.89,-0.34)
\lineto(0.83,-0.33)
\lineto(0.78,-0.32)
\lineto(0.69,-0.3)
\lineto(0.61,-0.28)
\lineto(0.54,-0.26)
\lineto(0.49,-0.24)
\lineto(0.43,-0.22)
\lineto(0.39,-0.2)
\lineto(0.35,-0.19)
\lineto(0.31,-0.17)
\lineto(0.28,-0.16)
\lineto(0.25,-0.14)
\lineto(0.22,-0.13)
\lineto(0.2,-0.12)
\lineto(0.17,-0.11)
\lineto(0.15,-0.09)
\lineto(0.13,-0.08)
\lineto(0.12,-0.07)
\lineto(0.1,-0.07)
\lineto(0.09,-0.06)
\lineto(0.07,-0.05)
\lineto(0.06,-0.04)
\lineto(0.05,-0.03)
\lineto(0.04,-0.03)
\lineto(0.03,-0.02)
\lineto(0.02,-0.01)
\lineto(0.01,-0.01)
\lineto(0,0)
\lineto(0,0)
\lineto(0,0)
\lineto(0,0)
\lineto(0,0)
\lineto(0,0)
\lineto(0,0)
\lineto(0,0)
\lineto(0,0)
\lineto(0,0)
\lineto(0,0)
\lineto(0,0)
\lineto(0,0)
\lineto(0,0)
\lineto(0,0)
\lineto(0,0)
\lineto(0,0)
\lineto(0,0)
\lineto(0,0)
\lineto(0,0)
\lineto(0,0)
\lineto(0,0)
\lineto(0,0)
\lineto(0,0)
\lineto(0,0)
\lineto(0,0)
\lineto(0,0)
\lineto(0,0)
\lineto(0,0)
\lineto(0,0)
\lineto(0,0)
\lineto(0,0)
\lineto(0,0)
\lineto(0,0)
\lineto(0,0)
\lineto(-0.01,0)
\lineto(-0.01,0.01)
\lineto(-0.02,0.01)
\lineto(-0.03,0.02)
\lineto(-0.04,0.03)
\lineto(-0.05,0.03)
\lineto(-0.05,0.04)
\lineto(-0.06,0.05)
\lineto(-0.07,0.05)
\lineto(-0.08,0.06)
\lineto(-0.09,0.07)
\lineto(-0.1,0.07)
\lineto(-0.11,0.08)
\lineto(-0.11,0.09)
\lineto(-0.12,0.1)
\lineto(-0.13,0.1)
\lineto(-0.14,0.11)
\lineto(-0.15,0.12)
\lineto(-0.16,0.13)
\lineto(-0.17,0.14)
\lineto(-0.17,0.14)
\lineto(-0.18,0.15)
\lineto(-0.19,0.16)
\lineto(-0.2,0.17)
\lineto(-0.21,0.18)
\lineto(-0.22,0.19)
\lineto(-0.23,0.2)
\lineto(-0.23,0.21)
\lineto(-0.24,0.22)
\lineto(-0.25,0.23)
\lineto(-0.26,0.24)
\lineto(-0.27,0.25)
\lineto(-0.28,0.26)
\lineto(-0.28,0.27)
\lineto(-0.29,0.28)
\lineto(-0.3,0.29)
\lineto(-0.31,0.3)
\lineto(-0.31,0.31)
\lineto(-0.32,0.32)
\lineto(-0.33,0.34)
\lineto(-0.34,0.35)
\lineto(-0.34,0.36)
\lineto(-0.35,0.37)
\lineto(-0.36,0.39)
\lineto(-0.36,0.4)
\lineto(-0.37,0.41)
\lineto(-0.38,0.42)
\lineto(-0.38,0.44)
\lineto(-0.39,0.45)
\lineto(-0.39,0.47)
\lineto(-0.4,0.48)
\lineto(-0.4,0.49)
\lineto(-0.41,0.51)
\lineto(-0.41,0.52)
\lineto(-0.42,0.54)
\lineto(-0.42,0.55)
\lineto(-0.42,0.57)
\lineto(-0.43,0.59)
\lineto(-0.43,0.6)
\lineto(-0.43,0.62)
\lineto(-0.44,0.64)
\lineto(-0.44,0.65)
\lineto(-0.44,0.67)
\lineto(-0.44,0.69)
\lineto(-0.44,0.7)
\lineto(-0.44,0.72)
\lineto(-0.44,0.74)
\lineto(-0.44,0.76)
\lineto(-0.44,0.79)
\lineto(-0.43,0.81)
\lineto(-0.43,0.84)
\lineto(-0.42,0.87)
\lineto(-0.41,0.91)
\lineto(-0.4,0.95)
\lineto(-0.38,0.99)
\lineto(-0.36,1.04)
\lineto(-0.36,1.04)
\lineto(-0.36,1.04)
}
\pscustom[linecolor=ccqqqq]{\moveto(0.22,1.56)
\lineto(0.22,1.56)
\lineto(0.27,1.58)
\lineto(0.33,1.61)
\lineto(0.39,1.63)
\lineto(0.45,1.66)
\lineto(0.52,1.68)
\lineto(0.59,1.7)
\lineto(0.67,1.72)
\lineto(0.75,1.74)
\lineto(0.84,1.76)
\lineto(0.94,1.78)
\lineto(0.98,1.78)
\lineto(1.04,1.79)
\lineto(1.09,1.8)
\lineto(1.14,1.8)
\lineto(1.2,1.8)
\lineto(1.25,1.81)
\lineto(1.31,1.81)
\lineto(1.37,1.81)
\lineto(1.43,1.81)
\lineto(1.49,1.81)
\lineto(1.56,1.81)
\lineto(1.62,1.8)
\lineto(1.68,1.8)
\lineto(1.75,1.79)
\lineto(1.82,1.78)
\lineto(1.88,1.77)
\lineto(1.95,1.76)
\lineto(2.02,1.74)
\lineto(2.09,1.73)
\lineto(2.16,1.71)
\lineto(2.22,1.69)
\lineto(2.29,1.67)
\lineto(2.36,1.64)
\lineto(2.42,1.62)
\lineto(2.49,1.59)
\lineto(2.55,1.56)
\lineto(2.61,1.53)
\lineto(2.67,1.49)
\lineto(2.73,1.46)
\lineto(2.78,1.42)
\lineto(2.84,1.38)
\lineto(2.89,1.34)
\lineto(2.98,1.26)
\lineto(3.05,1.16)
\lineto(3.12,1.07)
\lineto(3.16,0.97)
\lineto(3.19,0.87)
\lineto(3.21,0.78)
\lineto(3.21,0.68)
\lineto(3.2,0.58)
\lineto(3.18,0.49)
\lineto(3.14,0.4)
\lineto(3.09,0.32)
\lineto(3.04,0.24)
\lineto(2.97,0.17)
\lineto(2.9,0.1)
\lineto(2.83,0.04)
\lineto(2.75,-0.02)
\lineto(2.67,-0.07)
\lineto(2.59,-0.11)
\lineto(2.51,-0.15)
\lineto(2.49,-0.16)
\lineto(2.49,-0.16)
\lineto(2.49,-0.16)
\lineto(2.49,-0.16)
\lineto(2.49,-0.16)
\lineto(2.49,-0.16)
\lineto(2.49,-0.16)
\lineto(2.49,-0.16)
\lineto(2.49,-0.16)
\moveto(0.11,1.5)
\lineto(0.11,1.5)
\lineto(0.11,1.5)
\lineto(0.11,1.5)
\lineto(0.11,1.5)
\lineto(0.11,1.5)
\lineto(0.11,1.5)
\lineto(0.11,1.5)
\lineto(0.11,1.5)
\lineto(0.11,1.5)
\lineto(0.11,1.5)
\lineto(0.11,1.5)
\lineto(0.11,1.5)
\lineto(0.11,1.5)
\lineto(0.11,1.5)
\lineto(0.11,1.5)
\lineto(0.12,1.5)
\lineto(0.12,1.51)
\lineto(0.14,1.51)
\lineto(0.17,1.53)
\lineto(0.22,1.56)
\lineto(0.22,1.56)
}
\rput[tl](-0.96,2.86){$l:Mx+Ny+P=0$}
\rput[tl](3.22,1.52){$\psi(l)$}
\begin{scriptsize}
\psdots[dotstyle=*,linecolor=blue](0,0)
\rput[bl](-0.28,-0.35){\blue{$O$}}
\end{scriptsize}
\end{pspicture*}
	\end{center}
	\caption{Elliptic Inversion of the line $l$.}
	\label{fig:rectasinversionelipse2}
\end{figure}
Moreover,  it is clear that the ellipse passing through the center of inversion.
\end{proof}

\begin{corollary}
Let $l_1$ and $l_2$ be perpendicular lines intersecting at point  $P$. Then
\begin{enumerate}[i.]
    \item If $P\neq O$, then $\psi(l_1)$ and  $\psi(l_2)$ are orthogonal ellipses (their tangents at the points of intersection are perpendicular), which  pass through $P'$ and  $O$.
    \item If $P=O$, then $\psi(l_1)$ and  $\psi(l_2)$ are perpendicular lines.
    \item  If $l_1$ through $O$ but $l_2$ not through $O$, then $\psi(l_1)$ is an ellipse  and  $\psi(l_2)$ is an line which  passes through $O$ and it is orthogonal to $\psi(l_1)$ in $O$.
\end{enumerate}
\end{corollary}
\begin{proof}
\textit{i.}  Let $E$ be an ellipse of inversion with equation  $\frac{x^2}{a^2}+\frac{y^2}{b^2}=1$. Let   $l$ and $m$ be two perpendicular lines intersecting at $P$, ($P\neq O$), with respectively equations   $Mx+ Ny+P=0 $ ($P\neq0$) and $My-Nx+D=0$ ($D\neq 0$), see  Figure \ref{fig:RectaPerpenInvElip}.

\begin{figure}[h]
	\begin{center}
	\newrgbcolor{qqttzz}{0 0.2 0.6}
\newrgbcolor{ccqqqq}{0.8 0 0}
\newrgbcolor{xdxdff}{0.49 0.49 1}
\psset{xunit=0.7cm,yunit=0.7cm,algebraic=true,dotstyle=o,dotsize=3pt 0,linewidth=0.8pt,arrowsize=3pt 2,arrowinset=0.25}
\begin{pspicture*}(-8,-1.7)(5.48,3.44)
\rput{0}(0,0){\psellipse(0,0)(2.5,1.5)}
\rput[tl](-1.84,1.74){$E$}
\psplot[linecolor=qqttzz]{-6.06}{5.48}{(--7.75-3.44*x)/4.92}
\pscustom[linecolor=ccqqqq]{\moveto(-0.36,1.04)
\lineto(-0.33,1.09)
\lineto(-0.29,1.15)
\lineto(-0.24,1.22)
\lineto(-0.17,1.29)
\lineto(-0.07,1.38)
\lineto(-0.01,1.42)
\lineto(0.05,1.46)
\lineto(0.09,1.49)
\lineto(0.1,1.49)
\lineto(0.11,1.5)
\lineto(0.11,1.5)
\lineto(0.11,1.5)
\lineto(0.11,1.5)
\lineto(0.11,1.5)
\lineto(0.11,1.5)
\lineto(0.11,1.5)
\moveto(2.49,-0.16)
\lineto(2.49,-0.16)
\lineto(2.49,-0.16)
\lineto(2.49,-0.16)
\lineto(2.49,-0.16)
\lineto(2.49,-0.16)
\lineto(2.49,-0.16)
\lineto(2.49,-0.16)
\lineto(2.49,-0.16)
\lineto(2.49,-0.16)
\lineto(2.49,-0.16)
\lineto(2.48,-0.16)
\lineto(2.48,-0.16)
\lineto(2.48,-0.16)
\lineto(2.48,-0.16)
\lineto(2.48,-0.17)
\lineto(2.47,-0.17)
\lineto(2.46,-0.17)
\lineto(2.43,-0.19)
\lineto(2.38,-0.21)
\lineto(2.32,-0.23)
\lineto(2.27,-0.25)
\lineto(2.21,-0.26)
\lineto(2.16,-0.28)
\lineto(2.11,-0.29)
\lineto(2.06,-0.31)
\lineto(1.96,-0.33)
\lineto(1.87,-0.34)
\lineto(1.78,-0.36)
\lineto(1.69,-0.37)
\lineto(1.61,-0.37)
\lineto(1.54,-0.38)
\lineto(1.46,-0.38)
\lineto(1.39,-0.38)
\lineto(1.33,-0.38)
\lineto(1.27,-0.38)
\lineto(1.21,-0.38)
\lineto(1.16,-0.37)
\lineto(1.1,-0.37)
\lineto(1.01,-0.36)
\lineto(0.95,-0.35)
\lineto(0.89,-0.34)
\lineto(0.83,-0.33)
\lineto(0.78,-0.32)
\lineto(0.69,-0.3)
\lineto(0.61,-0.28)
\lineto(0.54,-0.26)
\lineto(0.49,-0.24)
\lineto(0.43,-0.22)
\lineto(0.39,-0.2)
\lineto(0.35,-0.19)
\lineto(0.31,-0.17)
\lineto(0.28,-0.16)
\lineto(0.25,-0.14)
\lineto(0.22,-0.13)
\lineto(0.2,-0.12)
\lineto(0.17,-0.11)
\lineto(0.15,-0.09)
\lineto(0.13,-0.08)
\lineto(0.12,-0.07)
\lineto(0.1,-0.07)
\lineto(0.09,-0.06)
\lineto(0.07,-0.05)
\lineto(0.06,-0.04)
\lineto(0.05,-0.03)
\lineto(0.04,-0.03)
\lineto(0.03,-0.02)
\lineto(0.02,-0.01)
\lineto(0.01,-0.01)
\lineto(0,0)
\lineto(0,0)
\lineto(0,0)
\lineto(0,0)
\lineto(0,0)
\lineto(0,0)
\lineto(0,0)
\lineto(0,0)
\lineto(0,0)
\lineto(0,0)
\lineto(0,0)
\lineto(0,0)
\lineto(0,0)
\lineto(0,0)
\lineto(0,0)
\lineto(0,0)
\lineto(0,0)
\lineto(0,0)
\lineto(0,0)
\lineto(0,0)
\lineto(0,0)
\lineto(0,0)
\lineto(0,0)
\lineto(0,0)
\lineto(0,0)
\lineto(0,0)
\lineto(0,0)
\lineto(0,0)
\lineto(0,0)
\lineto(0,0)
\lineto(0,0)
\lineto(0,0)
\lineto(0,0)
\lineto(0,0)
\lineto(0,0)
\lineto(-0.01,0)
\lineto(-0.01,0.01)
\lineto(-0.02,0.01)
\lineto(-0.03,0.02)
\lineto(-0.04,0.03)
\lineto(-0.05,0.03)
\lineto(-0.05,0.04)
\lineto(-0.06,0.05)
\lineto(-0.07,0.05)
\lineto(-0.08,0.06)
\lineto(-0.09,0.07)
\lineto(-0.1,0.07)
\lineto(-0.11,0.08)
\lineto(-0.11,0.09)
\lineto(-0.12,0.1)
\lineto(-0.13,0.1)
\lineto(-0.14,0.11)
\lineto(-0.15,0.12)
\lineto(-0.16,0.13)
\lineto(-0.17,0.14)
\lineto(-0.17,0.14)
\lineto(-0.18,0.15)
\lineto(-0.19,0.16)
\lineto(-0.2,0.17)
\lineto(-0.21,0.18)
\lineto(-0.22,0.19)
\lineto(-0.23,0.2)
\lineto(-0.23,0.21)
\lineto(-0.24,0.22)
\lineto(-0.25,0.23)
\lineto(-0.26,0.24)
\lineto(-0.27,0.25)
\lineto(-0.28,0.26)
\lineto(-0.28,0.27)
\lineto(-0.29,0.28)
\lineto(-0.3,0.29)
\lineto(-0.31,0.3)
\lineto(-0.31,0.31)
\lineto(-0.32,0.32)
\lineto(-0.33,0.34)
\lineto(-0.34,0.35)
\lineto(-0.34,0.36)
\lineto(-0.35,0.37)
\lineto(-0.36,0.39)
\lineto(-0.36,0.4)
\lineto(-0.37,0.41)
\lineto(-0.38,0.42)
\lineto(-0.38,0.44)
\lineto(-0.39,0.45)
\lineto(-0.39,0.47)
\lineto(-0.4,0.48)
\lineto(-0.4,0.49)
\lineto(-0.41,0.51)
\lineto(-0.41,0.52)
\lineto(-0.42,0.54)
\lineto(-0.42,0.55)
\lineto(-0.42,0.57)
\lineto(-0.43,0.59)
\lineto(-0.43,0.6)
\lineto(-0.43,0.62)
\lineto(-0.44,0.64)
\lineto(-0.44,0.65)
\lineto(-0.44,0.67)
\lineto(-0.44,0.69)
\lineto(-0.44,0.7)
\lineto(-0.44,0.72)
\lineto(-0.44,0.74)
\lineto(-0.44,0.76)
\lineto(-0.44,0.79)
\lineto(-0.43,0.81)
\lineto(-0.43,0.84)
\lineto(-0.42,0.87)
\lineto(-0.41,0.91)
\lineto(-0.4,0.95)
\lineto(-0.38,0.99)
\lineto(-0.36,1.04)
\lineto(-0.36,1.04)
\lineto(-0.36,1.04)
}
\pscustom[linecolor=ccqqqq]{\moveto(0.22,1.56)
\lineto(0.22,1.56)
\lineto(0.27,1.58)
\lineto(0.33,1.61)
\lineto(0.39,1.63)
\lineto(0.45,1.66)
\lineto(0.52,1.68)
\lineto(0.59,1.7)
\lineto(0.67,1.72)
\lineto(0.75,1.74)
\lineto(0.84,1.76)
\lineto(0.94,1.78)
\lineto(0.98,1.78)
\lineto(1.04,1.79)
\lineto(1.09,1.8)
\lineto(1.14,1.8)
\lineto(1.2,1.8)
\lineto(1.25,1.81)
\lineto(1.31,1.81)
\lineto(1.37,1.81)
\lineto(1.43,1.81)
\lineto(1.49,1.81)
\lineto(1.56,1.81)
\lineto(1.62,1.8)
\lineto(1.68,1.8)
\lineto(1.75,1.79)
\lineto(1.82,1.78)
\lineto(1.88,1.77)
\lineto(1.95,1.76)
\lineto(2.02,1.74)
\lineto(2.09,1.73)
\lineto(2.16,1.71)
\lineto(2.22,1.69)
\lineto(2.29,1.67)
\lineto(2.36,1.64)
\lineto(2.42,1.62)
\lineto(2.49,1.59)
\lineto(2.55,1.56)
\lineto(2.61,1.53)
\lineto(2.67,1.49)
\lineto(2.73,1.46)
\lineto(2.78,1.42)
\lineto(2.84,1.38)
\lineto(2.89,1.34)
\lineto(2.98,1.26)
\lineto(3.05,1.16)
\lineto(3.12,1.07)
\lineto(3.16,0.97)
\lineto(3.19,0.87)
\lineto(3.21,0.78)
\lineto(3.21,0.68)
\lineto(3.2,0.58)
\lineto(3.18,0.49)
\lineto(3.14,0.4)
\lineto(3.09,0.32)
\lineto(3.04,0.24)
\lineto(2.97,0.17)
\lineto(2.9,0.1)
\lineto(2.83,0.04)
\lineto(2.75,-0.02)
\lineto(2.67,-0.07)
\lineto(2.59,-0.11)
\lineto(2.51,-0.15)
\lineto(2.49,-0.16)
\lineto(2.49,-0.16)
\lineto(2.49,-0.16)
\lineto(2.49,-0.16)
\lineto(2.49,-0.16)
\lineto(2.49,-0.16)
\lineto(2.49,-0.16)
\lineto(2.49,-0.16)
\lineto(2.49,-0.16)
\moveto(0.11,1.5)
\lineto(0.11,1.5)
\lineto(0.11,1.5)
\lineto(0.11,1.5)
\lineto(0.11,1.5)
\lineto(0.11,1.5)
\lineto(0.11,1.5)
\lineto(0.11,1.5)
\lineto(0.11,1.5)
\lineto(0.11,1.5)
\lineto(0.11,1.5)
\lineto(0.11,1.5)
\lineto(0.11,1.5)
\lineto(0.11,1.5)
\lineto(0.11,1.5)
\lineto(0.11,1.5)
\lineto(0.12,1.5)
\lineto(0.12,1.51)
\lineto(0.14,1.51)
\lineto(0.17,1.53)
\lineto(0.22,1.56)
\lineto(0.22,1.56)
}
\rput[tl](-0.96,2.86){$l_1:Mx+Ny+P=0$}
\rput[tl](3.22,1.52){$\psi(l_1)$}
\psplot[linecolor=qqttzz]{-6.06}{5.48}{(--16.4--4.92*x)/3.44}
\pscustom[linecolor=ccqqqq]{\moveto(-1.94,0.14)
\lineto(-1.95,0.19)
\lineto(-1.95,0.25)
\lineto(-1.95,0.3)
\lineto(-1.93,0.36)
\lineto(-1.91,0.41)
\lineto(-1.88,0.46)
\lineto(-1.8,0.56)
\lineto(-1.71,0.64)
\lineto(-1.65,0.67)
\lineto(-1.6,0.7)
\lineto(-1.54,0.73)
\lineto(-1.48,0.75)
\lineto(-1.42,0.77)
\lineto(-1.37,0.79)
\lineto(-1.31,0.8)
\lineto(-1.26,0.82)
\lineto(-1.2,0.83)
\lineto(-1.15,0.83)
\lineto(-1.05,0.84)
\lineto(-0.96,0.85)
\lineto(-0.87,0.85)
\lineto(-0.8,0.84)
\lineto(-0.73,0.83)
\lineto(-0.67,0.82)
\lineto(-0.61,0.81)
\lineto(-0.56,0.8)
\lineto(-0.52,0.79)
\lineto(-0.48,0.78)
\lineto(-0.44,0.77)
\lineto(-0.4,0.76)
\lineto(-0.37,0.74)
\lineto(-0.34,0.73)
\lineto(-0.31,0.71)
\lineto(-0.28,0.7)
\lineto(-0.25,0.69)
\lineto(-0.22,0.67)
\lineto(-0.2,0.65)
\lineto(-0.17,0.64)
\lineto(-0.15,0.62)
\lineto(-0.13,0.61)
\lineto(-0.11,0.59)
\lineto(-0.09,0.57)
\lineto(-0.07,0.56)
\lineto(-0.06,0.54)
\lineto(-0.04,0.53)
\lineto(-0.03,0.51)
\lineto(-0.02,0.49)
\lineto(0,0.48)
\lineto(0.01,0.46)
\lineto(0.02,0.45)
\lineto(0.03,0.43)
\lineto(0.03,0.42)
\lineto(0.04,0.4)
\lineto(0.05,0.39)
\lineto(0.05,0.37)
\lineto(0.06,0.36)
\lineto(0.06,0.34)
\lineto(0.07,0.33)
\lineto(0.07,0.32)
\lineto(0.07,0.3)
\lineto(0.08,0.29)
\lineto(0.08,0.28)
\lineto(0.08,0.27)
\lineto(0.08,0.25)
\lineto(0.08,0.24)
\lineto(0.08,0.23)
\lineto(0.08,0.22)
\lineto(0.08,0.21)
\lineto(0.08,0.2)
\lineto(0.08,0.19)
\lineto(0.07,0.18)
\lineto(0.07,0.17)
\lineto(0.07,0.16)
\lineto(0.07,0.15)
\lineto(0.07,0.14)
\lineto(0.06,0.13)
\lineto(0.06,0.12)
\lineto(0.06,0.11)
\lineto(0.05,0.1)
\lineto(0.05,0.09)
\lineto(0.05,0.08)
\lineto(0.04,0.08)
\lineto(0.04,0.07)
\lineto(0.04,0.06)
\lineto(0.03,0.05)
\lineto(0.03,0.05)
\lineto(0.02,0.04)
\lineto(0.02,0.03)
\lineto(0.02,0.03)
\lineto(0.01,0.02)
\lineto(0.01,0.01)
\lineto(0,0.01)
\lineto(0,0)
\lineto(0,0)
\lineto(0,0)
\lineto(0,0)
\lineto(0,0)
\lineto(0,0)
\lineto(0,0)
\lineto(0,0)
\lineto(0,0)
\lineto(0,0)
\lineto(0,0)
\lineto(0,0)
\lineto(0,0)
\lineto(0,0)
\lineto(0,0)
\lineto(0,0)
\lineto(0,0)
\lineto(0,0)
\lineto(0,0)
\lineto(0,0)
\lineto(0,0)
\lineto(0,0)
\lineto(0,0)
\lineto(0,0)
\lineto(0,0)
\lineto(0,0)
\lineto(0,0)
\lineto(0,0)
\lineto(0,0)
\lineto(0,0)
\lineto(0,0)
\lineto(-0.01,-0.01)
\lineto(-0.01,-0.01)
\lineto(-0.02,-0.02)
\lineto(-0.02,-0.03)
\lineto(-0.03,-0.04)
\lineto(-0.03,-0.04)
\lineto(-0.04,-0.05)
\lineto(-0.05,-0.06)
\lineto(-0.06,-0.07)
\lineto(-0.07,-0.08)
\lineto(-0.08,-0.09)
\lineto(-0.09,-0.1)
\lineto(-0.1,-0.11)
\lineto(-0.11,-0.12)
\lineto(-0.12,-0.13)
\lineto(-0.14,-0.14)
\lineto(-0.16,-0.15)
\lineto(-0.18,-0.17)
\lineto(-0.2,-0.18)
\lineto(-0.22,-0.2)
\lineto(-0.24,-0.21)
\lineto(-0.27,-0.23)
\lineto(-0.3,-0.24)
\lineto(-0.34,-0.26)
\lineto(-0.38,-0.27)
\lineto(-0.42,-0.29)
\lineto(-0.47,-0.31)
\lineto(-0.53,-0.32)
\lineto(-0.59,-0.34)
\lineto(-0.66,-0.35)
\lineto(-0.73,-0.36)
\lineto(-0.82,-0.37)
\lineto(-0.92,-0.37)
\lineto(-0.97,-0.37)
\lineto(-1.02,-0.37)
\lineto(-1.08,-0.37)
\lineto(-1.14,-0.36)
\lineto(-1.2,-0.35)
\lineto(-1.26,-0.34)
\lineto(-1.32,-0.33)
\lineto(-1.39,-0.31)
\lineto(-1.45,-0.29)
\lineto(-1.51,-0.27)
\lineto(-1.58,-0.24)
\lineto(-1.64,-0.21)
\lineto(-1.7,-0.17)
\lineto(-1.75,-0.13)
\lineto(-1.84,-0.04)
\lineto(-1.88,0.01)
\lineto(-1.91,0.06)
\lineto(-1.94,0.14)
\lineto(-1.94,0.14)
}
\rput[tl](-7.5,1.68){$l_2:My-Nx+D=0$}
\rput[tl](-1.84,-0.4){$\psi(l_2)$}
\begin{scriptsize}
\psdots[dotstyle=*,linecolor=blue](0,0)
\rput[bl](-0.28,-0.4){\blue{$O$}}
\psdots[dotstyle=*,linecolor=qqttzz](-1.5,2.62)
\rput[bl](-1.6,2.9){\qqttzz{$P$}}
\psdots[dotstyle=*,linecolor=xdxdff](-0.44,0.77)
\rput[bl](-0.36,0.88){\xdxdff{$P'$}}
\end{scriptsize}
\end{pspicture*}
	\end{center}
	\caption{Inversion in an Ellipse of Perpendicular Lines.}
	\label{fig:RectaPerpenInvElip}
\end{figure}
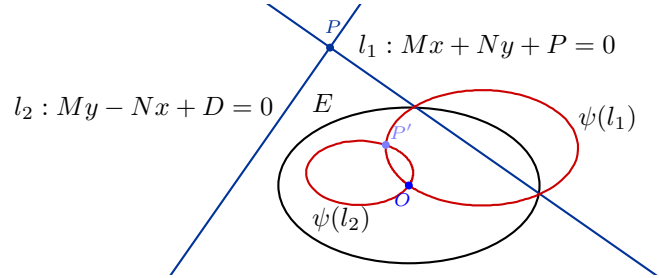

By Theorem  \ref{invrelipse}, $\psi(l_1)=l'_1$ and $\psi(l_2)=l'_2$ are ellipses pass through $O$ and their equations are
\begin{align*}
\frac{x^2}{a^2}+\frac{y^2}{b^2}+\frac{M}{P}x+\frac{N}{P}y&=0,\\
\frac{x^2}{a^2}+\frac{y^2}{b^2}-\frac{N}{D}x+\frac{M}{D}y&=0.
\end{align*}
The equations of the  tangent lines to these ellipses at  $O$, are:
\begin{align*}
\frac{Ma^2b^2}{2}x+\frac{Na^2b^2}{2}y&=0,\\
-\frac{Na^2b^2}{2}x+\frac{Ma^2b^2}{2}y&=0.
\end{align*}
Simplifying
\begin{align*}
Mx + Ny &=0,\\
-Nx + My &=0.
\end{align*}
Therefore  the lines are perpendicular and hence the ellipses are orthogonal.\\

\textit{ii.} It is clear by Theorem \ref{invrelipse}.\\

\textit{iii.} It is similar to the part 1.
\end{proof}

\begin{corollary}
The inversion in an ellipse of a system of concurrent lines for a point $H$, distinct of the  center of inversion is a set of coaxal system of  circles with two common points  $H'$ and the center of inversion, see Figure \ref{fig:RectasConcuInvElip}.
\end{corollary}
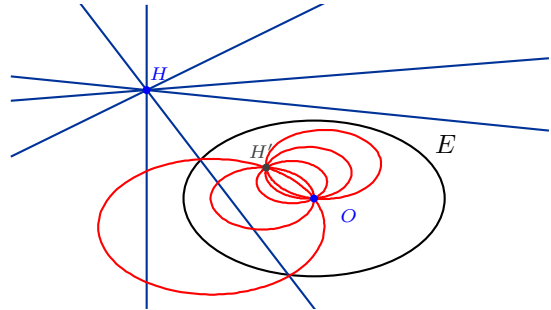
\begin{figure}[h]
	\begin{center}
\newrgbcolor{qqttzz}{0 0.2 0.6}
\psset{xunit=0.7cm,yunit=0.7cm,algebraic=true,dotstyle=o,dotsize=3pt 0,linewidth=0.8pt,arrowsize=3pt 2,arrowinset=0.25}
\begin{pspicture*}(-5.76,-2.1)(4.46,3.68)
\rput{0}(0,0){\psellipse(0,0)(2.5,1.5)}
\rput[tl](2.3,1.2){$E$}
\psplot[linecolor=qqttzz]{-5.76}{4.46}{(-7.38-1*x)/-2.04}
\psplot[linecolor=qqttzz]{-5.76}{4.46}{(--11.7--0.4*x)/5.06}
\psline[linecolor=qqttzz](-3.18,-2.1)(-3.18,3.68)
\psplot[linecolor=qqttzz]{-5.76}{4.46}{(--4.93--3.08*x)/-2.36}
\psplot[linecolor=qqttzz]{-5.76}{4.46}{(-8.28--0.48*x)/-4.76}
\pscustom[linecolor=red]{\moveto(0.2,-0.53)
\lineto(0.19,-0.6)
\lineto(0.18,-0.67)
\lineto(0.17,-0.76)
\lineto(0.13,-0.86)
\lineto(0.1,-0.91)
\lineto(0.07,-0.97)
\lineto(0.03,-1.03)
\lineto(-0.02,-1.1)
\lineto(-0.08,-1.17)
\lineto(-0.16,-1.25)
\lineto(-0.26,-1.33)
\lineto(-0.31,-1.37)
\lineto(-0.37,-1.41)
\lineto(-0.44,-1.45)
\lineto(-0.46,-1.46)
\lineto(-0.47,-1.47)
\lineto(-0.47,-1.47)
\lineto(-0.47,-1.47)
\lineto(-0.47,-1.47)
\lineto(-0.47,-1.47)
\lineto(-0.47,-1.47)
\lineto(-0.47,-1.47)
\lineto(-0.47,-1.47)
\moveto(-2.17,0.74)
\lineto(-2.17,0.74)
\lineto(-2.17,0.74)
\lineto(-2.17,0.74)
\lineto(-2.17,0.74)
\lineto(-2.17,0.74)
\lineto(-2.17,0.74)
\lineto(-2.17,0.74)
\lineto(-2.17,0.74)
\lineto(-2.17,0.74)
\lineto(-2.17,0.74)
\lineto(-2.17,0.74)
\lineto(-2.17,0.74)
\lineto(-2.17,0.74)
\lineto(-2.17,0.74)
\lineto(-2.16,0.74)
\lineto(-2.16,0.74)
\lineto(-2.14,0.75)
\lineto(-2.11,0.75)
\lineto(-2.06,0.75)
\lineto(-2,0.75)
\lineto(-1.95,0.75)
\lineto(-1.86,0.75)
\lineto(-1.8,0.75)
\lineto(-1.74,0.74)
\lineto(-1.68,0.74)
\lineto(-1.63,0.74)
\lineto(-1.54,0.73)
\lineto(-1.46,0.72)
\lineto(-1.38,0.7)
\lineto(-1.32,0.69)
\lineto(-1.26,0.68)
\lineto(-1.21,0.67)
\lineto(-1.16,0.66)
\lineto(-1.12,0.65)
\lineto(-1.08,0.64)
\lineto(-1.05,0.63)
\lineto(-1.02,0.62)
\lineto(-0.99,0.61)
\lineto(-0.96,0.6)
\lineto(-0.93,0.6)
\lineto(-0.91,0.59)
\lineto(-0.89,0.58)
\lineto(-0.87,0.57)
\lineto(-0.85,0.57)
\lineto(-0.82,0.56)
\lineto(-0.8,0.55)
\lineto(-0.78,0.54)
\lineto(-0.76,0.53)
\lineto(-0.74,0.53)
\lineto(-0.72,0.52)
\lineto(-0.7,0.51)
\lineto(-0.68,0.5)
\lineto(-0.66,0.49)
\lineto(-0.65,0.48)
\lineto(-0.63,0.48)
\lineto(-0.61,0.47)
\lineto(-0.59,0.46)
\lineto(-0.57,0.45)
\lineto(-0.55,0.44)
\lineto(-0.54,0.43)
\lineto(-0.52,0.42)
\lineto(-0.5,0.41)
\lineto(-0.49,0.4)
\lineto(-0.47,0.4)
\lineto(-0.45,0.39)
\lineto(-0.44,0.38)
\lineto(-0.42,0.37)
\lineto(-0.41,0.36)
\lineto(-0.39,0.35)
\lineto(-0.38,0.34)
\lineto(-0.36,0.33)
\lineto(-0.35,0.32)
\lineto(-0.33,0.31)
\lineto(-0.32,0.3)
\lineto(-0.31,0.29)
\lineto(-0.29,0.28)
\lineto(-0.28,0.27)
\lineto(-0.27,0.26)
\lineto(-0.26,0.25)
\lineto(-0.24,0.24)
\lineto(-0.23,0.23)
\lineto(-0.22,0.22)
\lineto(-0.21,0.21)
\lineto(-0.2,0.2)
\lineto(-0.18,0.19)
\lineto(-0.17,0.18)
\lineto(-0.16,0.17)
\lineto(-0.15,0.16)
\lineto(-0.14,0.16)
\lineto(-0.13,0.15)
\lineto(-0.12,0.14)
\lineto(-0.11,0.13)
\lineto(-0.1,0.12)
\lineto(-0.09,0.11)
\lineto(-0.08,0.1)
\lineto(-0.07,0.09)
\lineto(-0.06,0.08)
\lineto(-0.06,0.07)
\lineto(-0.05,0.06)
\lineto(-0.04,0.05)
\lineto(-0.03,0.04)
\lineto(-0.02,0.03)
\lineto(-0.02,0.02)
\lineto(-0.01,0.01)
\lineto(0,0)
\lineto(0,0)
\lineto(0,0)
\lineto(0,0)
\lineto(0,0)
\lineto(0,0)
\lineto(0,0)
\lineto(0,0)
\lineto(0,0)
\lineto(0,0)
\lineto(0,0)
\lineto(0,0)
\lineto(0,0)
\lineto(0,0)
\lineto(0,0)
\lineto(0,0)
\lineto(0,0)
\lineto(0,0)
\lineto(0,0)
\lineto(0,0)
\lineto(0,0)
\lineto(0,0)
\lineto(0,0)
\lineto(0,0)
\lineto(0,0)
\lineto(0,0)
\lineto(0,0)
\lineto(0,0)
\lineto(0,-0.01)
\lineto(0.01,-0.01)
\lineto(0.02,-0.02)
\lineto(0.03,-0.03)
\lineto(0.03,-0.05)
\lineto(0.04,-0.06)
\lineto(0.05,-0.07)
\lineto(0.06,-0.09)
\lineto(0.07,-0.1)
\lineto(0.08,-0.12)
\lineto(0.09,-0.14)
\lineto(0.1,-0.16)
\lineto(0.11,-0.18)
\lineto(0.12,-0.21)
\lineto(0.14,-0.23)
\lineto(0.15,-0.26)
\lineto(0.16,-0.29)
\lineto(0.17,-0.33)
\lineto(0.18,-0.37)
\lineto(0.19,-0.42)
\lineto(0.19,-0.47)
\lineto(0.2,-0.52)
\lineto(0.2,-0.53)
}
\pscustom[linecolor=red]{\moveto(-0.69,-1.58)
\lineto(-0.69,-1.58)
\lineto(-0.78,-1.62)
\lineto(-0.84,-1.64)
\lineto(-0.89,-1.66)
\lineto(-0.95,-1.68)
\lineto(-1.01,-1.69)
\lineto(-1.07,-1.71)
\lineto(-1.13,-1.73)
\lineto(-1.2,-1.74)
\lineto(-1.26,-1.76)
\lineto(-1.33,-1.77)
\lineto(-1.41,-1.78)
\lineto(-1.48,-1.8)
\lineto(-1.56,-1.81)
\lineto(-1.64,-1.81)
\lineto(-1.72,-1.82)
\lineto(-1.81,-1.82)
\lineto(-1.89,-1.83)
\lineto(-1.98,-1.83)
\lineto(-2.07,-1.82)
\lineto(-2.16,-1.82)
\lineto(-2.25,-1.81)
\lineto(-2.35,-1.8)
\lineto(-2.44,-1.79)
\lineto(-2.53,-1.78)
\lineto(-2.63,-1.76)
\lineto(-2.72,-1.74)
\lineto(-2.81,-1.72)
\lineto(-2.9,-1.69)
\lineto(-3,-1.66)
\lineto(-3.08,-1.63)
\lineto(-3.17,-1.6)
\lineto(-3.26,-1.56)
\lineto(-3.34,-1.52)
\lineto(-3.42,-1.48)
\lineto(-3.49,-1.43)
\lineto(-3.57,-1.39)
\lineto(-3.63,-1.34)
\lineto(-3.7,-1.29)
\lineto(-3.76,-1.24)
\lineto(-3.81,-1.18)
\lineto(-3.86,-1.13)
\lineto(-3.91,-1.07)
\lineto(-3.95,-1.01)
\lineto(-3.98,-0.96)
\lineto(-4.01,-0.9)
\lineto(-4.04,-0.84)
\lineto(-4.06,-0.78)
\lineto(-4.08,-0.72)
\lineto(-4.09,-0.66)
\lineto(-4.1,-0.61)
\lineto(-4.1,-0.55)
\lineto(-4.1,-0.49)
\lineto(-4.09,-0.44)
\lineto(-4.08,-0.39)
\lineto(-4.07,-0.33)
\lineto(-4.06,-0.28)
\lineto(-4.04,-0.23)
\lineto(-3.99,-0.14)
\lineto(-3.94,-0.05)
\lineto(-3.88,0.03)
\lineto(-3.81,0.11)
\lineto(-3.74,0.18)
\lineto(-3.66,0.24)
\lineto(-3.58,0.3)
\lineto(-3.5,0.35)
\lineto(-3.42,0.4)
\lineto(-3.34,0.44)
\lineto(-3.26,0.48)
\lineto(-3.18,0.52)
\lineto(-3.11,0.55)
\lineto(-3.03,0.58)
\lineto(-2.96,0.6)
\lineto(-2.89,0.62)
\lineto(-2.82,0.64)
\lineto(-2.75,0.66)
\lineto(-2.69,0.67)
\lineto(-2.62,0.69)
\lineto(-2.56,0.7)
\lineto(-2.5,0.71)
\lineto(-2.45,0.72)
\lineto(-2.39,0.72)
\lineto(-2.34,0.73)
\lineto(-2.24,0.74)
\lineto(-2.2,0.74)
\lineto(-2.17,0.74)
\lineto(-2.17,0.74)
\lineto(-2.17,0.74)
\lineto(-2.17,0.74)
\lineto(-2.17,0.74)
\lineto(-2.17,0.74)
\lineto(-2.17,0.74)
\moveto(-0.47,-1.47)
\lineto(-0.47,-1.47)
\lineto(-0.47,-1.47)
\lineto(-0.47,-1.47)
\lineto(-0.47,-1.47)
\lineto(-0.47,-1.47)
\lineto(-0.47,-1.47)
\lineto(-0.47,-1.47)
\lineto(-0.47,-1.47)
\lineto(-0.47,-1.47)
\lineto(-0.47,-1.47)
\lineto(-0.47,-1.47)
\lineto(-0.47,-1.47)
\lineto(-0.48,-1.47)
\lineto(-0.48,-1.48)
\lineto(-0.49,-1.48)
\lineto(-0.5,-1.49)
\lineto(-0.52,-1.5)
\lineto(-0.58,-1.53)
\lineto(-0.64,-1.56)
\lineto(-0.69,-1.58)
\lineto(-0.69,-1.58)
}
\pscustom[linecolor=red]{\moveto(-1.51,0.5)
\lineto(-1.57,0.47)
\lineto(-1.63,0.44)
\lineto(-1.69,0.41)
\lineto(-1.74,0.37)
\lineto(-1.8,0.33)
\lineto(-1.85,0.27)
\lineto(-1.9,0.21)
\lineto(-1.93,0.15)
\lineto(-1.96,0.07)
\lineto(-1.97,-0.01)
\lineto(-1.95,-0.09)
\lineto(-1.92,-0.18)
\lineto(-1.86,-0.26)
\lineto(-1.78,-0.34)
\lineto(-1.74,-0.38)
\lineto(-1.68,-0.41)
\lineto(-1.63,-0.45)
\lineto(-1.56,-0.48)
\lineto(-1.5,-0.5)
\lineto(-1.43,-0.52)
\lineto(-1.37,-0.54)
\lineto(-1.3,-0.56)
\lineto(-1.23,-0.57)
\lineto(-1.16,-0.58)
\lineto(-1.09,-0.59)
\lineto(-1.02,-0.59)
\lineto(-0.96,-0.59)
\lineto(-0.89,-0.59)
\lineto(-0.83,-0.58)
\lineto(-0.77,-0.58)
\lineto(-0.72,-0.57)
\lineto(-0.67,-0.56)
\lineto(-0.57,-0.53)
\lineto(-0.48,-0.51)
\lineto(-0.41,-0.48)
\lineto(-0.34,-0.45)
\lineto(-0.29,-0.42)
\lineto(-0.24,-0.38)
\lineto(-0.2,-0.35)
\lineto(-0.16,-0.32)
\lineto(-0.13,-0.3)
\lineto(-0.11,-0.27)
\lineto(-0.09,-0.24)
\lineto(-0.07,-0.22)
\lineto(-0.05,-0.19)
\lineto(-0.04,-0.17)
\lineto(-0.03,-0.15)
\lineto(-0.02,-0.13)
\lineto(-0.02,-0.11)
\lineto(-0.01,-0.09)
\lineto(-0.01,-0.08)
\lineto(-0.01,-0.06)
\lineto(0,-0.05)
\lineto(0,-0.03)
\lineto(0,-0.02)
\lineto(0,-0.01)
\lineto(0,0)
\lineto(0,0)
\lineto(0,0)
\lineto(0,0)
\lineto(0,0)
\lineto(0,0)
\lineto(0,0)
\lineto(0,0)
\lineto(0,0)
\lineto(0,0)
\lineto(0,0)
\lineto(0,0)
\lineto(0,0)
\lineto(0,0)
\lineto(0,0)
\lineto(0,0)
\lineto(0,0)
\lineto(0,0)
\lineto(0,0)
\lineto(0,0)
\lineto(0,0)
\lineto(0,0)
\lineto(0,0)
\lineto(0,0)
\lineto(0,0)
\lineto(0,0)
\lineto(0,0)
\lineto(0,0)
\lineto(0,0)
\lineto(0,0)
\lineto(0,0.01)
\lineto(0,0.02)
\lineto(0,0.03)
\lineto(0,0.04)
\lineto(0,0.05)
\lineto(-0.01,0.07)
\lineto(-0.01,0.08)
\lineto(-0.01,0.09)
\lineto(-0.02,0.1)
\lineto(-0.02,0.12)
\lineto(-0.02,0.13)
\lineto(-0.03,0.14)
\lineto(-0.03,0.15)
\lineto(-0.04,0.17)
\lineto(-0.05,0.18)
\lineto(-0.05,0.19)
\lineto(-0.06,0.2)
\lineto(-0.07,0.22)
\lineto(-0.08,0.23)
\lineto(-0.09,0.24)
\lineto(-0.1,0.26)
\lineto(-0.11,0.27)
\lineto(-0.12,0.28)
\lineto(-0.13,0.29)
\lineto(-0.14,0.3)
\lineto(-0.15,0.32)
\lineto(-0.17,0.33)
\lineto(-0.18,0.34)
\lineto(-0.19,0.35)
\lineto(-0.21,0.36)
\lineto(-0.22,0.37)
\lineto(-0.24,0.39)
\lineto(-0.26,0.4)
\lineto(-0.27,0.41)
\lineto(-0.29,0.42)
\lineto(-0.31,0.43)
\lineto(-0.32,0.44)
\lineto(-0.34,0.45)
\lineto(-0.36,0.46)
\lineto(-0.38,0.47)
\lineto(-0.4,0.47)
\lineto(-0.42,0.48)
\lineto(-0.44,0.49)
\lineto(-0.46,0.5)
\lineto(-0.48,0.51)
\lineto(-0.5,0.51)
\lineto(-0.52,0.52)
\lineto(-0.54,0.53)
\lineto(-0.56,0.53)
\lineto(-0.58,0.54)
\lineto(-0.61,0.54)
\lineto(-0.63,0.55)
\lineto(-0.65,0.55)
\lineto(-0.67,0.56)
\lineto(-0.69,0.56)
\lineto(-0.72,0.57)
\lineto(-0.74,0.57)
\lineto(-0.76,0.57)
\lineto(-0.78,0.58)
\lineto(-0.81,0.58)
\lineto(-0.83,0.58)
\lineto(-0.85,0.58)
\lineto(-0.87,0.59)
\lineto(-0.89,0.59)
\lineto(-0.92,0.59)
\lineto(-0.94,0.59)
\lineto(-0.96,0.59)
\lineto(-0.99,0.59)
\lineto(-1.02,0.59)
\lineto(-1.04,0.59)
\lineto(-1.07,0.59)
\lineto(-1.11,0.58)
\lineto(-1.14,0.58)
\lineto(-1.18,0.58)
\lineto(-1.22,0.57)
\lineto(-1.26,0.57)
\lineto(-1.3,0.56)
\lineto(-1.35,0.55)
\lineto(-1.4,0.53)
\lineto(-1.45,0.52)
\lineto(-1.51,0.5)
\lineto(-1.51,0.5)
}
\pscustom[linecolor=red]{\moveto(-1,0.11)
\lineto(-0.98,0.09)
\lineto(-0.95,0.07)
\lineto(-0.93,0.05)
\lineto(-0.9,0.03)
\lineto(-0.87,0.01)
\lineto(-0.83,-0.01)
\lineto(-0.8,-0.02)
\lineto(-0.76,-0.04)
\lineto(-0.72,-0.05)
\lineto(-0.68,-0.06)
\lineto(-0.63,-0.07)
\lineto(-0.59,-0.08)
\lineto(-0.54,-0.08)
\lineto(-0.5,-0.09)
\lineto(-0.45,-0.09)
\lineto(-0.41,-0.09)
\lineto(-0.36,-0.09)
\lineto(-0.32,-0.09)
\lineto(-0.27,-0.08)
\lineto(-0.23,-0.07)
\lineto(-0.19,-0.07)
\lineto(-0.15,-0.06)
\lineto(-0.11,-0.05)
\lineto(-0.08,-0.03)
\lineto(-0.05,-0.02)
\lineto(-0.02,-0.01)
\lineto(0,0)
\lineto(0,0)
\lineto(0,0)
\lineto(0,0)
\lineto(0,0)
\lineto(0,0)
\lineto(0,0)
\lineto(0,0)
\lineto(0,0)
\lineto(0,0)
\lineto(0,0)
\lineto(0,0)
\lineto(0,0)
\lineto(0,0)
\lineto(0,0)
\lineto(0,0)
\lineto(0,0)
\lineto(0,0)
\lineto(0,0)
\lineto(0,0)
\lineto(0,0)
\lineto(0,0)
\lineto(0,0)
\lineto(0,0)
\lineto(0,0)
\lineto(0,0)
\lineto(0,0)
\lineto(0,0)
\lineto(0.01,0)
\lineto(0.02,0.01)
\lineto(0.04,0.02)
\lineto(0.07,0.04)
\lineto(0.09,0.06)
\lineto(0.12,0.08)
\lineto(0.14,0.1)
\lineto(0.17,0.12)
\lineto(0.19,0.14)
\lineto(0.2,0.17)
\lineto(0.22,0.2)
\lineto(0.23,0.23)
\lineto(0.24,0.25)
\lineto(0.24,0.28)
\lineto(0.25,0.32)
\lineto(0.24,0.35)
\lineto(0.24,0.38)
\lineto(0.23,0.41)
\lineto(0.21,0.44)
\lineto(0.19,0.47)
\lineto(0.17,0.49)
\lineto(0.15,0.52)
\lineto(0.12,0.54)
\lineto(0.09,0.57)
\lineto(0.06,0.59)
\lineto(0.02,0.61)
\lineto(-0.01,0.63)
\lineto(-0.05,0.64)
\lineto(-0.09,0.66)
\lineto(-0.12,0.67)
\lineto(-0.16,0.68)
\lineto(-0.2,0.69)
\lineto(-0.24,0.7)
\lineto(-0.27,0.7)
\lineto(-0.31,0.71)
\lineto(-0.34,0.71)
\lineto(-0.38,0.71)
\lineto(-0.41,0.71)
\lineto(-0.44,0.71)
\lineto(-0.47,0.71)
\lineto(-0.5,0.71)
\lineto(-0.53,0.71)
\lineto(-0.56,0.7)
\lineto(-0.58,0.7)
\lineto(-0.61,0.7)
\lineto(-0.63,0.69)
\lineto(-0.65,0.69)
\lineto(-0.67,0.68)
\lineto(-0.69,0.68)
\lineto(-0.71,0.67)
\lineto(-0.73,0.67)
\lineto(-0.75,0.66)
\lineto(-0.76,0.66)
\lineto(-0.78,0.65)
\lineto(-0.79,0.65)
\lineto(-0.81,0.64)
\lineto(-0.82,0.63)
\lineto(-0.83,0.63)
\lineto(-0.84,0.62)
\lineto(-0.85,0.62)
\lineto(-0.86,0.61)
\lineto(-0.87,0.61)
\lineto(-0.88,0.6)
\lineto(-0.89,0.6)
\lineto(-0.9,0.59)
\lineto(-0.91,0.59)
\lineto(-0.92,0.58)
\lineto(-0.93,0.58)
\lineto(-0.94,0.57)
\lineto(-0.94,0.56)
\lineto(-0.95,0.56)
\lineto(-0.96,0.55)
\lineto(-0.97,0.54)
\lineto(-0.98,0.54)
\lineto(-0.99,0.53)
\lineto(-0.99,0.52)
\lineto(-1,0.51)
\lineto(-1.01,0.5)
\lineto(-1.02,0.49)
\lineto(-1.03,0.48)
\lineto(-1.04,0.47)
\lineto(-1.04,0.46)
\lineto(-1.05,0.45)
\lineto(-1.06,0.44)
\lineto(-1.06,0.43)
\lineto(-1.07,0.41)
\lineto(-1.08,0.4)
\lineto(-1.08,0.39)
\lineto(-1.08,0.37)
\lineto(-1.09,0.36)
\lineto(-1.09,0.34)
\lineto(-1.09,0.32)
\lineto(-1.09,0.31)
\lineto(-1.09,0.29)
\lineto(-1.09,0.27)
\lineto(-1.08,0.25)
\lineto(-1.08,0.23)
\lineto(-1.07,0.21)
\lineto(-1.06,0.19)
\lineto(-1.05,0.17)
\lineto(-1.04,0.15)
\lineto(-1.02,0.13)
\lineto(-1,0.11)
\lineto(-1,0.11)
}
\pscustom[linecolor=red]{\moveto(-0.28,0.97)
\lineto(-0.2,0.97)
\lineto(-0.12,0.98)
\lineto(-0.04,0.98)
\lineto(0.05,0.97)
\lineto(0.14,0.95)
\lineto(0.23,0.93)
\lineto(0.32,0.91)
\lineto(0.4,0.87)
\lineto(0.47,0.83)
\lineto(0.54,0.79)
\lineto(0.59,0.74)
\lineto(0.64,0.69)
\lineto(0.67,0.64)
\lineto(0.7,0.58)
\lineto(0.71,0.53)
\lineto(0.71,0.48)
\lineto(0.71,0.43)
\lineto(0.69,0.38)
\lineto(0.67,0.34)
\lineto(0.65,0.3)
\lineto(0.62,0.27)
\lineto(0.59,0.23)
\lineto(0.56,0.2)
\lineto(0.53,0.18)
\lineto(0.49,0.15)
\lineto(0.46,0.13)
\lineto(0.42,0.11)
\lineto(0.39,0.1)
\lineto(0.35,0.08)
\lineto(0.32,0.07)
\lineto(0.29,0.06)
\lineto(0.26,0.05)
\lineto(0.23,0.04)
\lineto(0.2,0.03)
\lineto(0.17,0.03)
\lineto(0.15,0.02)
\lineto(0.12,0.02)
\lineto(0.1,0.01)
\lineto(0.08,0.01)
\lineto(0.05,0.01)
\lineto(0.03,0)
\lineto(0.01,0)
\lineto(0,0)
\lineto(0,0)
\lineto(0,0)
\lineto(0,0)
\lineto(0,0)
\lineto(0,0)
\lineto(0,0)
\lineto(0,0)
\lineto(0,0)
\lineto(0,0)
\lineto(0,0)
\lineto(0,0)
\lineto(0,0)
\lineto(0,0)
\lineto(0,0)
\lineto(0,0)
\lineto(0,0)
\lineto(0,0)
\lineto(0,0)
\lineto(0,0)
\lineto(0,0)
\lineto(0,0)
\lineto(0,0)
\lineto(0,0)
\lineto(0,0)
\lineto(0,0)
\lineto(0,0)
\lineto(0,0)
\lineto(0,0)
\lineto(0,0)
\lineto(-0.01,0)
\lineto(-0.01,0)
\lineto(-0.03,0)
\lineto(-0.05,0)
\lineto(-0.07,0)
\lineto(-0.09,0)
\lineto(-0.11,0)
\lineto(-0.13,0)
\lineto(-0.15,0)
\lineto(-0.17,0)
\lineto(-0.19,0)
\lineto(-0.21,0)
\lineto(-0.23,0)
\lineto(-0.26,0)
\lineto(-0.28,0.01)
\lineto(-0.3,0.01)
\lineto(-0.32,0.01)
\lineto(-0.35,0.02)
\lineto(-0.37,0.02)
\lineto(-0.39,0.03)
\lineto(-0.41,0.03)
\lineto(-0.44,0.04)
\lineto(-0.46,0.04)
\lineto(-0.48,0.05)
\lineto(-0.5,0.06)
\lineto(-0.52,0.06)
\lineto(-0.55,0.07)
\lineto(-0.57,0.08)
\lineto(-0.59,0.09)
\lineto(-0.61,0.1)
\lineto(-0.63,0.11)
\lineto(-0.65,0.12)
\lineto(-0.67,0.13)
\lineto(-0.69,0.14)
\lineto(-0.71,0.15)
\lineto(-0.72,0.16)
\lineto(-0.74,0.18)
\lineto(-0.76,0.19)
\lineto(-0.77,0.2)
\lineto(-0.79,0.22)
\lineto(-0.8,0.23)
\lineto(-0.82,0.24)
\lineto(-0.83,0.26)
\lineto(-0.84,0.27)
\lineto(-0.85,0.29)
\lineto(-0.86,0.3)
\lineto(-0.87,0.31)
\lineto(-0.88,0.33)
\lineto(-0.89,0.34)
\lineto(-0.9,0.36)
\lineto(-0.9,0.37)
\lineto(-0.91,0.39)
\lineto(-0.91,0.4)
\lineto(-0.92,0.42)
\lineto(-0.92,0.43)
\lineto(-0.92,0.45)
\lineto(-0.92,0.46)
\lineto(-0.93,0.48)
\lineto(-0.92,0.49)
\lineto(-0.92,0.51)
\lineto(-0.92,0.52)
\lineto(-0.92,0.54)
\lineto(-0.92,0.55)
\lineto(-0.91,0.57)
\lineto(-0.91,0.58)
\lineto(-0.91,0.59)
\lineto(-0.9,0.61)
\lineto(-0.89,0.62)
\lineto(-0.89,0.64)
\lineto(-0.88,0.65)
\lineto(-0.87,0.67)
\lineto(-0.85,0.69)
\lineto(-0.84,0.7)
\lineto(-0.82,0.72)
\lineto(-0.81,0.74)
\lineto(-0.79,0.76)
\lineto(-0.76,0.78)
\lineto(-0.73,0.8)
\lineto(-0.7,0.82)
\lineto(-0.67,0.84)
\lineto(-0.63,0.87)
\lineto(-0.58,0.89)
\lineto(-0.53,0.91)
\lineto(-0.48,0.92)
\lineto(-0.42,0.94)
\lineto(-0.35,0.96)
\lineto(-0.28,0.97)
}
\pscustom[linecolor=red]{\moveto(-0.67,0.24)
\lineto(-0.66,0.23)
\lineto(-0.64,0.22)
\lineto(-0.63,0.2)
\lineto(-0.61,0.19)
\lineto(-0.59,0.18)
\lineto(-0.58,0.17)
\lineto(-0.56,0.17)
\lineto(-0.54,0.16)
\lineto(-0.53,0.15)
\lineto(-0.51,0.14)
\lineto(-0.49,0.13)
\lineto(-0.47,0.12)
\lineto(-0.45,0.11)
\lineto(-0.43,0.1)
\lineto(-0.42,0.1)
\lineto(-0.4,0.09)
\lineto(-0.38,0.08)
\lineto(-0.36,0.08)
\lineto(-0.34,0.07)
\lineto(-0.32,0.06)
\lineto(-0.3,0.06)
\lineto(-0.28,0.05)
\lineto(-0.26,0.05)
\lineto(-0.24,0.04)
\lineto(-0.22,0.04)
\lineto(-0.2,0.03)
\lineto(-0.18,0.03)
\lineto(-0.16,0.02)
\lineto(-0.14,0.02)
\lineto(-0.12,0.02)
\lineto(-0.1,0.01)
\lineto(-0.08,0.01)
\lineto(-0.06,0.01)
\lineto(-0.04,0)
\lineto(-0.02,0)
\lineto(-0.01,0)
\lineto(0,0)
\lineto(0,0)
\lineto(0,0)
\lineto(0,0)
\lineto(0,0)
\lineto(0,0)
\lineto(0,0)
\lineto(0,0)
\lineto(0,0)
\lineto(0,0)
\lineto(0,0)
\lineto(0,0)
\lineto(0,0)
\lineto(0,0)
\lineto(0,0)
\lineto(0,0)
\lineto(0,0)
\lineto(0,0)
\lineto(0,0)
\lineto(0,0)
\lineto(0,0)
\lineto(0,0)
\lineto(0,0)
\lineto(0,0)
\lineto(0,0)
\lineto(0,0)
\lineto(0,0)
\lineto(0,0)
\lineto(0,0)
\lineto(0.01,0)
\lineto(0.01,0)
\lineto(0.03,0)
\lineto(0.05,0)
\lineto(0.07,-0.01)
\lineto(0.1,-0.01)
\lineto(0.13,-0.01)
\lineto(0.15,-0.01)
\lineto(0.19,-0.01)
\lineto(0.22,-0.01)
\lineto(0.25,-0.01)
\lineto(0.29,-0.01)
\lineto(0.33,0)
\lineto(0.37,0)
\lineto(0.42,0.01)
\lineto(0.47,0.01)
\lineto(0.52,0.02)
\lineto(0.58,0.04)
\lineto(0.64,0.05)
\lineto(0.7,0.07)
\lineto(0.77,0.09)
\lineto(0.84,0.12)
\lineto(0.91,0.16)
\lineto(0.98,0.2)
\lineto(1.05,0.25)
\lineto(1.12,0.31)
\lineto(1.18,0.39)
\lineto(1.23,0.47)
\lineto(1.26,0.56)
\lineto(1.27,0.61)
\lineto(1.27,0.66)
\lineto(1.27,0.71)
\lineto(1.26,0.77)
\lineto(1.23,0.82)
\lineto(1.21,0.88)
\lineto(1.17,0.93)
\lineto(1.12,0.98)
\lineto(1.07,1.03)
\lineto(1,1.08)
\lineto(0.93,1.12)
\lineto(0.86,1.16)
\lineto(0.78,1.2)
\lineto(0.69,1.23)
\lineto(0.61,1.25)
\lineto(0.52,1.27)
\lineto(0.43,1.29)
\lineto(0.34,1.3)
\lineto(0.25,1.3)
\lineto(0.17,1.3)
\lineto(0.08,1.3)
\lineto(0,1.29)
\lineto(-0.07,1.28)
\lineto(-0.14,1.27)
\lineto(-0.21,1.26)
\lineto(-0.27,1.24)
\lineto(-0.33,1.23)
\lineto(-0.38,1.21)
\lineto(-0.48,1.17)
\lineto(-0.56,1.13)
\lineto(-0.63,1.09)
\lineto(-0.69,1.05)
\lineto(-0.73,1.01)
\lineto(-0.77,0.97)
\lineto(-0.8,0.93)
\lineto(-0.83,0.9)
\lineto(-0.85,0.87)
\lineto(-0.86,0.84)
\lineto(-0.88,0.81)
\lineto(-0.89,0.78)
\lineto(-0.9,0.76)
\lineto(-0.9,0.74)
\lineto(-0.91,0.71)
\lineto(-0.91,0.69)
\lineto(-0.91,0.67)
\lineto(-0.91,0.66)
\lineto(-0.91,0.64)
\lineto(-0.91,0.62)
\lineto(-0.91,0.61)
\lineto(-0.91,0.6)
\lineto(-0.91,0.58)
\lineto(-0.9,0.57)
\lineto(-0.9,0.56)
\lineto(-0.9,0.54)
\lineto(-0.89,0.53)
\lineto(-0.89,0.52)
\lineto(-0.89,0.51)
\lineto(-0.88,0.49)
\lineto(-0.88,0.48)
\lineto(-0.87,0.47)
\lineto(-0.86,0.45)
\lineto(-0.86,0.44)
\lineto(-0.85,0.43)
\lineto(-0.84,0.41)
\lineto(-0.83,0.4)
\lineto(-0.82,0.39)
\lineto(-0.81,0.38)
\lineto(-0.81,0.36)
\lineto(-0.8,0.35)
\lineto(-0.78,0.34)
\lineto(-0.77,0.33)
\lineto(-0.76,0.32)
\lineto(-0.75,0.3)
\lineto(-0.74,0.29)
\lineto(-0.72,0.28)
\lineto(-0.71,0.27)
\lineto(-0.7,0.26)
\lineto(-0.68,0.25)
\lineto(-0.67,0.24)
}
\begin{scriptsize}
\psdots[dotstyle=*,linecolor=blue](0,0)
\rput[bl](0.5,-0.44){\blue{$O$}}
\psdots[dotstyle=*,linecolor=blue](-3.18,2.06)
\rput[bl](-3.12,2.26){\blue{$H$}}
\psdots[dotstyle=*,linecolor=darkgray](-0.91,0.59)
\rput[bl](-1.24,0.76){\darkgray{$H'$}}
\end{scriptsize}
\end{pspicture*}
	\end{center}
	\caption{Inversion in an Ellipse of a System of Concurrent Lines.}
	\label{fig:RectasConcuInvElip}
\end{figure}

\begin{corollary}
The inversion in an ellipse of a system of parallel lines which does not pass through  of the center of inversion is a set of tangent ellipses at the center of inversion, see Figure \ref{fig:RectasParaleInvElipse}.
\end{corollary}

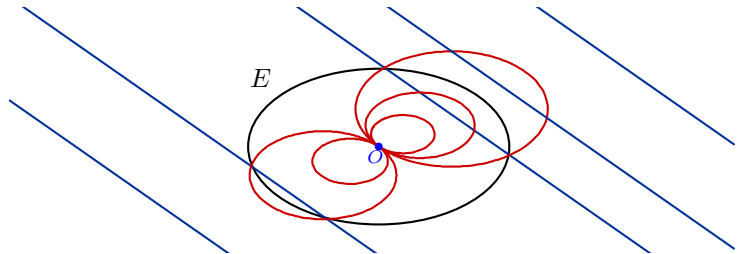
\begin{figure}[h]
	\begin{center}
\newrgbcolor{qqttzz}{0 0.2 0.6}
\newrgbcolor{ccqqqq}{0.8 0 0}
\psset{xunit=0.7cm,yunit=0.7cm,algebraic=true,dotstyle=o,dotsize=3pt 0,linewidth=0.8pt,arrowsize=3pt 2,arrowinset=0.25}
\begin{pspicture*}(-7.02,-2.02)(6.76,2.64)
\rput{0}(0,0){\psellipse(0,0)(2.5,1.5)}
\rput[tl](-2.44,1.48){$E$}
\psplot[linecolor=qqttzz]{-7.02}{6.76}{(--7.75-3.44*x)/4.92}
\pscustom[linecolor=ccqqqq]{\moveto(-0.36,1.04)
\lineto(-0.33,1.09)
\lineto(-0.29,1.15)
\lineto(-0.24,1.22)
\lineto(-0.17,1.29)
\lineto(-0.07,1.38)
\lineto(-0.01,1.42)
\lineto(0.05,1.46)
\lineto(0.09,1.49)
\lineto(0.1,1.49)
\lineto(0.11,1.5)
\lineto(0.11,1.5)
\lineto(0.11,1.5)
\lineto(0.11,1.5)
\lineto(0.11,1.5)
\lineto(0.11,1.5)
\lineto(0.11,1.5)
\moveto(2.49,-0.16)
\lineto(2.49,-0.16)
\lineto(2.49,-0.16)
\lineto(2.49,-0.16)
\lineto(2.49,-0.16)
\lineto(2.49,-0.16)
\lineto(2.49,-0.16)
\lineto(2.49,-0.16)
\lineto(2.49,-0.16)
\lineto(2.49,-0.16)
\lineto(2.49,-0.16)
\lineto(2.48,-0.16)
\lineto(2.48,-0.16)
\lineto(2.48,-0.16)
\lineto(2.48,-0.16)
\lineto(2.48,-0.17)
\lineto(2.47,-0.17)
\lineto(2.46,-0.17)
\lineto(2.43,-0.19)
\lineto(2.38,-0.21)
\lineto(2.32,-0.23)
\lineto(2.27,-0.25)
\lineto(2.21,-0.26)
\lineto(2.16,-0.28)
\lineto(2.11,-0.29)
\lineto(2.06,-0.31)
\lineto(1.96,-0.33)
\lineto(1.87,-0.34)
\lineto(1.78,-0.36)
\lineto(1.69,-0.37)
\lineto(1.61,-0.37)
\lineto(1.54,-0.38)
\lineto(1.46,-0.38)
\lineto(1.39,-0.38)
\lineto(1.33,-0.38)
\lineto(1.27,-0.38)
\lineto(1.21,-0.38)
\lineto(1.16,-0.37)
\lineto(1.1,-0.37)
\lineto(1.01,-0.36)
\lineto(0.95,-0.35)
\lineto(0.89,-0.34)
\lineto(0.83,-0.33)
\lineto(0.78,-0.32)
\lineto(0.69,-0.3)
\lineto(0.61,-0.28)
\lineto(0.54,-0.26)
\lineto(0.49,-0.24)
\lineto(0.43,-0.22)
\lineto(0.39,-0.2)
\lineto(0.35,-0.19)
\lineto(0.31,-0.17)
\lineto(0.28,-0.16)
\lineto(0.25,-0.14)
\lineto(0.22,-0.13)
\lineto(0.2,-0.12)
\lineto(0.17,-0.11)
\lineto(0.15,-0.09)
\lineto(0.13,-0.08)
\lineto(0.12,-0.07)
\lineto(0.1,-0.07)
\lineto(0.09,-0.06)
\lineto(0.07,-0.05)
\lineto(0.06,-0.04)
\lineto(0.05,-0.03)
\lineto(0.04,-0.03)
\lineto(0.03,-0.02)
\lineto(0.02,-0.01)
\lineto(0.01,-0.01)
\lineto(0,0)
\lineto(0,0)
\lineto(0,0)
\lineto(0,0)
\lineto(0,0)
\lineto(0,0)
\lineto(0,0)
\lineto(0,0)
\lineto(0,0)
\lineto(0,0)
\lineto(0,0)
\lineto(0,0)
\lineto(0,0)
\lineto(0,0)
\lineto(0,0)
\lineto(0,0)
\lineto(0,0)
\lineto(0,0)
\lineto(0,0)
\lineto(0,0)
\lineto(0,0)
\lineto(0,0)
\lineto(0,0)
\lineto(0,0)
\lineto(0,0)
\lineto(0,0)
\lineto(0,0)
\lineto(0,0)
\lineto(0,0)
\lineto(0,0)
\lineto(0,0)
\lineto(0,0)
\lineto(0,0)
\lineto(0,0)
\lineto(0,0)
\lineto(-0.01,0)
\lineto(-0.01,0.01)
\lineto(-0.02,0.01)
\lineto(-0.03,0.02)
\lineto(-0.04,0.03)
\lineto(-0.05,0.03)
\lineto(-0.05,0.04)
\lineto(-0.06,0.05)
\lineto(-0.07,0.05)
\lineto(-0.08,0.06)
\lineto(-0.09,0.07)
\lineto(-0.1,0.07)
\lineto(-0.11,0.08)
\lineto(-0.11,0.09)
\lineto(-0.12,0.1)
\lineto(-0.13,0.1)
\lineto(-0.14,0.11)
\lineto(-0.15,0.12)
\lineto(-0.16,0.13)
\lineto(-0.17,0.14)
\lineto(-0.17,0.14)
\lineto(-0.18,0.15)
\lineto(-0.19,0.16)
\lineto(-0.2,0.17)
\lineto(-0.21,0.18)
\lineto(-0.22,0.19)
\lineto(-0.23,0.2)
\lineto(-0.23,0.21)
\lineto(-0.24,0.22)
\lineto(-0.25,0.23)
\lineto(-0.26,0.24)
\lineto(-0.27,0.25)
\lineto(-0.28,0.26)
\lineto(-0.28,0.27)
\lineto(-0.29,0.28)
\lineto(-0.3,0.29)
\lineto(-0.31,0.3)
\lineto(-0.31,0.31)
\lineto(-0.32,0.32)
\lineto(-0.33,0.34)
\lineto(-0.34,0.35)
\lineto(-0.34,0.36)
\lineto(-0.35,0.37)
\lineto(-0.36,0.39)
\lineto(-0.36,0.4)
\lineto(-0.37,0.41)
\lineto(-0.38,0.42)
\lineto(-0.38,0.44)
\lineto(-0.39,0.45)
\lineto(-0.39,0.47)
\lineto(-0.4,0.48)
\lineto(-0.4,0.49)
\lineto(-0.41,0.51)
\lineto(-0.41,0.52)
\lineto(-0.42,0.54)
\lineto(-0.42,0.55)
\lineto(-0.42,0.57)
\lineto(-0.43,0.59)
\lineto(-0.43,0.6)
\lineto(-0.43,0.62)
\lineto(-0.44,0.64)
\lineto(-0.44,0.65)
\lineto(-0.44,0.67)
\lineto(-0.44,0.69)
\lineto(-0.44,0.7)
\lineto(-0.44,0.72)
\lineto(-0.44,0.74)
\lineto(-0.44,0.76)
\lineto(-0.44,0.79)
\lineto(-0.43,0.81)
\lineto(-0.43,0.84)
\lineto(-0.42,0.87)
\lineto(-0.41,0.91)
\lineto(-0.4,0.95)
\lineto(-0.38,0.99)
\lineto(-0.36,1.04)
\lineto(-0.36,1.04)
\lineto(-0.36,1.04)
}
\pscustom[linecolor=ccqqqq]{\moveto(0.22,1.56)
\lineto(0.22,1.56)
\lineto(0.27,1.58)
\lineto(0.33,1.61)
\lineto(0.39,1.63)
\lineto(0.45,1.66)
\lineto(0.52,1.68)
\lineto(0.59,1.7)
\lineto(0.67,1.72)
\lineto(0.75,1.74)
\lineto(0.84,1.76)
\lineto(0.94,1.78)
\lineto(0.98,1.78)
\lineto(1.04,1.79)
\lineto(1.09,1.8)
\lineto(1.14,1.8)
\lineto(1.2,1.8)
\lineto(1.25,1.81)
\lineto(1.31,1.81)
\lineto(1.37,1.81)
\lineto(1.43,1.81)
\lineto(1.49,1.81)
\lineto(1.56,1.81)
\lineto(1.62,1.8)
\lineto(1.68,1.8)
\lineto(1.75,1.79)
\lineto(1.82,1.78)
\lineto(1.88,1.77)
\lineto(1.95,1.76)
\lineto(2.02,1.74)
\lineto(2.09,1.73)
\lineto(2.16,1.71)
\lineto(2.22,1.69)
\lineto(2.29,1.67)
\lineto(2.36,1.64)
\lineto(2.42,1.62)
\lineto(2.49,1.59)
\lineto(2.55,1.56)
\lineto(2.61,1.53)
\lineto(2.67,1.49)
\lineto(2.73,1.46)
\lineto(2.78,1.42)
\lineto(2.84,1.38)
\lineto(2.89,1.34)
\lineto(2.98,1.26)
\lineto(3.05,1.16)
\lineto(3.12,1.07)
\lineto(3.16,0.97)
\lineto(3.19,0.87)
\lineto(3.21,0.78)
\lineto(3.21,0.68)
\lineto(3.2,0.58)
\lineto(3.18,0.49)
\lineto(3.14,0.4)
\lineto(3.09,0.32)
\lineto(3.04,0.24)
\lineto(2.97,0.17)
\lineto(2.9,0.1)
\lineto(2.83,0.04)
\lineto(2.75,-0.02)
\lineto(2.67,-0.07)
\lineto(2.59,-0.11)
\lineto(2.51,-0.15)
\lineto(2.49,-0.16)
\lineto(2.49,-0.16)
\lineto(2.49,-0.16)
\lineto(2.49,-0.16)
\lineto(2.49,-0.16)
\lineto(2.49,-0.16)
\lineto(2.49,-0.16)
\lineto(2.49,-0.16)
\lineto(2.49,-0.16)
\moveto(0.11,1.5)
\lineto(0.11,1.5)
\lineto(0.11,1.5)
\lineto(0.11,1.5)
\lineto(0.11,1.5)
\lineto(0.11,1.5)
\lineto(0.11,1.5)
\lineto(0.11,1.5)
\lineto(0.11,1.5)
\lineto(0.11,1.5)
\lineto(0.11,1.5)
\lineto(0.11,1.5)
\lineto(0.11,1.5)
\lineto(0.11,1.5)
\lineto(0.11,1.5)
\lineto(0.11,1.5)
\lineto(0.12,1.5)
\lineto(0.12,1.51)
\lineto(0.14,1.51)
\lineto(0.17,1.53)
\lineto(0.22,1.56)
\lineto(0.22,1.56)
}
\psplot[linecolor=qqttzz]{-7.02}{6.76}{(-10.16-3.44*x)/4.92}
\psplot[linecolor=qqttzz]{-7.02}{6.76}{(-19.8-3.44*x)/4.92}
\psplot[linecolor=qqttzz]{-7.02}{6.76}{(--23.42-3.44*x)/4.92}
\pscustom[linecolor=ccqqqq]{\moveto(-0.08,0.4)
\lineto(-0.08,0.41)
\lineto(-0.07,0.42)
\lineto(-0.06,0.42)
\lineto(-0.05,0.43)
\lineto(-0.04,0.44)
\lineto(-0.03,0.45)
\lineto(-0.02,0.46)
\lineto(-0.01,0.46)
\lineto(0,0.47)
\lineto(0.01,0.48)
\lineto(0.02,0.48)
\lineto(0.03,0.49)
\lineto(0.04,0.5)
\lineto(0.05,0.5)
\lineto(0.06,0.51)
\lineto(0.07,0.52)
\lineto(0.09,0.52)
\lineto(0.1,0.53)
\lineto(0.11,0.53)
\lineto(0.12,0.54)
\lineto(0.14,0.54)
\lineto(0.15,0.55)
\lineto(0.17,0.55)
\lineto(0.19,0.56)
\lineto(0.21,0.57)
\lineto(0.23,0.57)
\lineto(0.25,0.58)
\lineto(0.27,0.58)
\lineto(0.3,0.59)
\lineto(0.33,0.59)
\lineto(0.36,0.59)
\lineto(0.39,0.6)
\lineto(0.43,0.6)
\lineto(0.47,0.6)
\lineto(0.51,0.6)
\lineto(0.56,0.59)
\lineto(0.6,0.59)
\lineto(0.65,0.58)
\lineto(0.71,0.57)
\lineto(0.76,0.55)
\lineto(0.81,0.53)
\lineto(0.86,0.51)
\lineto(0.92,0.47)
\lineto(0.96,0.44)
\lineto(1,0.4)
\lineto(1.03,0.35)
\lineto(1.05,0.3)
\lineto(1.06,0.25)
\lineto(1.06,0.2)
\lineto(1.05,0.15)
\lineto(1.02,0.1)
\lineto(0.98,0.06)
\lineto(0.94,0.02)
\lineto(0.89,-0.02)
\lineto(0.83,-0.05)
\lineto(0.77,-0.07)
\lineto(0.71,-0.09)
\lineto(0.66,-0.11)
\lineto(0.6,-0.12)
\lineto(0.55,-0.12)
\lineto(0.49,-0.13)
\lineto(0.45,-0.13)
\lineto(0.4,-0.12)
\lineto(0.36,-0.12)
\lineto(0.32,-0.12)
\lineto(0.29,-0.11)
\lineto(0.26,-0.11)
\lineto(0.23,-0.1)
\lineto(0.2,-0.09)
\lineto(0.18,-0.08)
\lineto(0.15,-0.08)
\lineto(0.13,-0.07)
\lineto(0.11,-0.06)
\lineto(0.1,-0.05)
\lineto(0.08,-0.05)
\lineto(0.07,-0.04)
\lineto(0.05,-0.03)
\lineto(0.04,-0.03)
\lineto(0.03,-0.02)
\lineto(0.02,-0.01)
\lineto(0.01,-0.01)
\lineto(0,0)
\lineto(0,0)
\lineto(0,0)
\lineto(0,0)
\lineto(0,0)
\lineto(0,0)
\lineto(0,0)
\lineto(0,0)
\lineto(0,0)
\lineto(0,0)
\lineto(0,0)
\lineto(0,0)
\lineto(0,0)
\lineto(0,0)
\lineto(0,0)
\lineto(0,0)
\lineto(0,0)
\lineto(0,0)
\lineto(0,0)
\lineto(0,0)
\lineto(0,0)
\lineto(0,0)
\lineto(0,0)
\lineto(0,0)
\lineto(0,0)
\lineto(0,0)
\lineto(0,0)
\lineto(0,0)
\lineto(0,0)
\lineto(0,0)
\lineto(0,0)
\lineto(-0.01,0)
\lineto(-0.01,0.01)
\lineto(-0.02,0.01)
\lineto(-0.03,0.02)
\lineto(-0.04,0.03)
\lineto(-0.04,0.03)
\lineto(-0.05,0.04)
\lineto(-0.06,0.05)
\lineto(-0.07,0.06)
\lineto(-0.07,0.06)
\lineto(-0.08,0.07)
\lineto(-0.09,0.08)
\lineto(-0.09,0.09)
\lineto(-0.1,0.09)
\lineto(-0.1,0.1)
\lineto(-0.11,0.11)
\lineto(-0.11,0.12)
\lineto(-0.12,0.13)
\lineto(-0.12,0.14)
\lineto(-0.13,0.15)
\lineto(-0.13,0.15)
\lineto(-0.13,0.16)
\lineto(-0.14,0.17)
\lineto(-0.14,0.18)
\lineto(-0.14,0.19)
\lineto(-0.14,0.2)
\lineto(-0.14,0.21)
\lineto(-0.15,0.22)
\lineto(-0.15,0.23)
\lineto(-0.15,0.24)
\lineto(-0.15,0.25)
\lineto(-0.14,0.26)
\lineto(-0.14,0.27)
\lineto(-0.14,0.28)
\lineto(-0.14,0.29)
\lineto(-0.14,0.3)
\lineto(-0.13,0.31)
\lineto(-0.13,0.32)
\lineto(-0.13,0.33)
\lineto(-0.12,0.34)
\lineto(-0.12,0.35)
\lineto(-0.11,0.36)
\lineto(-0.11,0.37)
\lineto(-0.1,0.38)
\lineto(-0.09,0.38)
\lineto(-0.09,0.39)
\lineto(-0.08,0.4)
\lineto(-0.08,0.4)
}
\pscustom[linecolor=ccqqqq]{\moveto(-0.95,-1.38)
\lineto(-0.95,-1.38)
\lineto(-0.89,-1.37)
\lineto(-0.82,-1.37)
\lineto(-0.77,-1.36)
\lineto(-0.71,-1.35)
\lineto(-0.65,-1.34)
\lineto(-0.6,-1.33)
\lineto(-0.54,-1.32)
\lineto(-0.49,-1.31)
\lineto(-0.39,-1.28)
\lineto(-0.3,-1.25)
\lineto(-0.22,-1.21)
\lineto(-0.14,-1.18)
\lineto(-0.08,-1.14)
\lineto(-0.02,-1.1)
\lineto(0.04,-1.06)
\lineto(0.13,-0.98)
\lineto(0.2,-0.91)
\lineto(0.25,-0.84)
\lineto(0.29,-0.77)
\lineto(0.31,-0.7)
\lineto(0.33,-0.65)
\lineto(0.33,-0.59)
\lineto(0.34,-0.54)
\lineto(0.33,-0.5)
\lineto(0.33,-0.46)
\lineto(0.32,-0.42)
\lineto(0.31,-0.38)
\lineto(0.3,-0.35)
\lineto(0.29,-0.32)
\lineto(0.27,-0.29)
\lineto(0.26,-0.27)
\lineto(0.24,-0.25)
\lineto(0.23,-0.23)
\lineto(0.22,-0.21)
\lineto(0.2,-0.19)
\lineto(0.19,-0.17)
\lineto(0.18,-0.15)
\lineto(0.16,-0.14)
\lineto(0.15,-0.13)
\lineto(0.14,-0.11)
\lineto(0.12,-0.1)
\lineto(0.11,-0.09)
\lineto(0.1,-0.08)
\lineto(0.09,-0.07)
\lineto(0.08,-0.06)
\lineto(0.07,-0.05)
\lineto(0.06,-0.04)
\lineto(0.05,-0.04)
\lineto(0.04,-0.03)
\lineto(0.03,-0.02)
\lineto(0.02,-0.01)
\lineto(0.01,-0.01)
\lineto(0,0)
\lineto(0,0)
\lineto(0,0)
\lineto(0,0)
\lineto(0,0)
\lineto(0,0)
\lineto(0,0)
\lineto(0,0)
\lineto(0,0)
\lineto(0,0)
\lineto(0,0)
\lineto(0,0)
\lineto(0,0)
\lineto(0,0)
\lineto(0,0)
\lineto(0,0)
\lineto(0,0)
\lineto(0,0)
\lineto(0,0)
\lineto(0,0)
\lineto(0,0)
\lineto(0,0)
\lineto(0,0)
\lineto(0,0)
\lineto(0,0)
\lineto(0,0)
\lineto(0,0)
\lineto(-0.01,0)
\lineto(-0.01,0.01)
\lineto(-0.02,0.01)
\lineto(-0.03,0.02)
\lineto(-0.04,0.03)
\lineto(-0.05,0.03)
\lineto(-0.06,0.04)
\lineto(-0.07,0.05)
\lineto(-0.08,0.05)
\lineto(-0.09,0.06)
\lineto(-0.1,0.07)
\lineto(-0.12,0.07)
\lineto(-0.13,0.08)
\lineto(-0.14,0.09)
\lineto(-0.15,0.09)
\lineto(-0.17,0.1)
\lineto(-0.18,0.11)
\lineto(-0.2,0.11)
\lineto(-0.21,0.12)
\lineto(-0.23,0.13)
\lineto(-0.25,0.14)
\lineto(-0.26,0.14)
\lineto(-0.28,0.15)
\lineto(-0.3,0.16)
\lineto(-0.32,0.16)
\lineto(-0.34,0.17)
\lineto(-0.36,0.18)
\lineto(-0.38,0.19)
\lineto(-0.4,0.19)
\lineto(-0.43,0.2)
\lineto(-0.45,0.21)
\lineto(-0.48,0.22)
\lineto(-0.5,0.22)
\lineto(-0.53,0.23)
\lineto(-0.56,0.24)
\lineto(-0.59,0.24)
\lineto(-0.62,0.25)
\lineto(-0.65,0.25)
\lineto(-0.68,0.26)
\lineto(-0.71,0.27)
\lineto(-0.75,0.27)
\lineto(-0.78,0.28)
\lineto(-0.82,0.28)
\lineto(-0.86,0.28)
\lineto(-0.9,0.29)
\lineto(-0.94,0.29)
\lineto(-0.98,0.29)
\lineto(-1.03,0.29)
\lineto(-1.07,0.29)
\lineto(-1.12,0.29)
\lineto(-1.17,0.29)
\lineto(-1.22,0.29)
\lineto(-1.27,0.28)
\lineto(-1.32,0.28)
\lineto(-1.37,0.27)
\lineto(-1.42,0.26)
\lineto(-1.48,0.25)
\lineto(-1.53,0.24)
\lineto(-1.59,0.23)
\lineto(-1.65,0.21)
\lineto(-1.7,0.2)
\lineto(-1.76,0.18)
\lineto(-1.82,0.16)
\lineto(-1.87,0.13)
\lineto(-1.93,0.11)
\lineto(-1.98,0.08)
\lineto(-2.04,0.05)
\lineto(-2.11,0.01)
\lineto(-2.17,-0.04)
\lineto(-2.23,-0.09)
\lineto(-2.29,-0.15)
\lineto(-2.34,-0.22)
\lineto(-2.39,-0.3)
\lineto(-2.41,-0.35)
\lineto(-2.42,-0.37)
\lineto(-2.42,-0.37)
\lineto(-2.42,-0.37)
\lineto(-2.42,-0.37)
\lineto(-2.42,-0.37)
\lineto(-2.42,-0.37)
\lineto(-2.42,-0.37)
\moveto(-0.98,-1.38)
\lineto(-0.98,-1.38)
\lineto(-0.98,-1.38)
\lineto(-0.98,-1.38)
\lineto(-0.98,-1.38)
\lineto(-0.98,-1.38)
\lineto(-0.98,-1.38)
\lineto(-0.98,-1.38)
\lineto(-0.98,-1.38)
\lineto(-0.98,-1.38)
\lineto(-0.98,-1.38)
\lineto(-0.98,-1.38)
\lineto(-0.98,-1.38)
\lineto(-0.98,-1.38)
\lineto(-0.98,-1.38)
\lineto(-0.98,-1.38)
\lineto(-0.97,-1.38)
\lineto(-0.96,-1.38)
\lineto(-0.95,-1.38)
}
\pscustom[linecolor=ccqqqq]{\moveto(-0.73,0.13)
\lineto(-0.75,0.13)
\lineto(-0.77,0.13)
\lineto(-0.79,0.12)
\lineto(-0.81,0.12)
\lineto(-0.83,0.11)
\lineto(-0.85,0.11)
\lineto(-0.87,0.1)
\lineto(-0.89,0.1)
\lineto(-0.91,0.09)
\lineto(-0.93,0.08)
\lineto(-0.95,0.07)
\lineto(-0.96,0.07)
\lineto(-0.98,0.06)
\lineto(-1,0.05)
\lineto(-1.01,0.04)
\lineto(-1.03,0.03)
\lineto(-1.05,0.03)
\lineto(-1.06,0.02)
\lineto(-1.08,0.01)
\lineto(-1.09,0)
\lineto(-1.11,-0.02)
\lineto(-1.12,-0.03)
\lineto(-1.14,-0.04)
\lineto(-1.15,-0.06)
\lineto(-1.17,-0.07)
\lineto(-1.19,-0.09)
\lineto(-1.2,-0.11)
\lineto(-1.22,-0.13)
\lineto(-1.23,-0.16)
\lineto(-1.24,-0.19)
\lineto(-1.25,-0.21)
\lineto(-1.26,-0.25)
\lineto(-1.26,-0.28)
\lineto(-1.26,-0.32)
\lineto(-1.25,-0.35)
\lineto(-1.23,-0.39)
\lineto(-1.21,-0.44)
\lineto(-1.18,-0.48)
\lineto(-1.13,-0.52)
\lineto(-1.08,-0.56)
\lineto(-1.01,-0.6)
\lineto(-0.94,-0.64)
\lineto(-0.85,-0.67)
\lineto(-0.75,-0.69)
\lineto(-0.66,-0.7)
\lineto(-0.6,-0.71)
\lineto(-0.55,-0.71)
\lineto(-0.5,-0.71)
\lineto(-0.45,-0.71)
\lineto(-0.35,-0.69)
\lineto(-0.26,-0.67)
\lineto(-0.18,-0.65)
\lineto(-0.1,-0.62)
\lineto(-0.04,-0.58)
\lineto(0.01,-0.55)
\lineto(0.06,-0.51)
\lineto(0.09,-0.48)
\lineto(0.12,-0.44)
\lineto(0.14,-0.4)
\lineto(0.16,-0.37)
\lineto(0.17,-0.34)
\lineto(0.17,-0.31)
\lineto(0.17,-0.28)
\lineto(0.17,-0.26)
\lineto(0.17,-0.23)
\lineto(0.16,-0.21)
\lineto(0.16,-0.19)
\lineto(0.15,-0.17)
\lineto(0.14,-0.15)
\lineto(0.13,-0.14)
\lineto(0.12,-0.12)
\lineto(0.11,-0.11)
\lineto(0.1,-0.09)
\lineto(0.09,-0.08)
\lineto(0.08,-0.07)
\lineto(0.07,-0.06)
\lineto(0.06,-0.05)
\lineto(0.05,-0.04)
\lineto(0.04,-0.03)
\lineto(0.03,-0.03)
\lineto(0.02,-0.02)
\lineto(0.02,-0.01)
\lineto(0.01,-0.01)
\lineto(0,0)
\lineto(-0.01,0)
\lineto(-0.01,0.01)
\lineto(-0.02,0.01)
\lineto(-0.03,0.02)
\lineto(-0.04,0.03)
\lineto(-0.05,0.03)
\lineto(-0.06,0.04)
\lineto(-0.07,0.04)
\lineto(-0.08,0.05)
\lineto(-0.09,0.05)
\lineto(-0.1,0.06)
\lineto(-0.12,0.07)
\lineto(-0.13,0.07)
\lineto(-0.14,0.08)
\lineto(-0.15,0.08)
\lineto(-0.17,0.09)
\lineto(-0.18,0.09)
\lineto(-0.2,0.1)
\lineto(-0.21,0.1)
\lineto(-0.23,0.11)
\lineto(-0.24,0.11)
\lineto(-0.26,0.11)
\lineto(-0.27,0.12)
\lineto(-0.29,0.12)
\lineto(-0.31,0.13)
\lineto(-0.32,0.13)
\lineto(-0.34,0.13)
\lineto(-0.36,0.14)
\lineto(-0.38,0.14)
\lineto(-0.4,0.14)
\lineto(-0.41,0.14)
\lineto(-0.43,0.14)
\lineto(-0.45,0.15)
\lineto(-0.47,0.15)
\lineto(-0.49,0.15)
\lineto(-0.51,0.15)
\lineto(-0.53,0.15)
\lineto(-0.55,0.15)
\lineto(-0.57,0.15)
\lineto(-0.59,0.15)
\lineto(-0.61,0.15)
\lineto(-0.63,0.15)
\lineto(-0.65,0.14)
\lineto(-0.67,0.14)
\lineto(-0.7,0.14)
\lineto(-0.72,0.14)
\lineto(-0.73,0.13)
}
\pscustom[linecolor=ccqqqq]{\moveto(-2.36,-0.85)
\lineto(-2.31,-0.91)
\lineto(-2.26,-0.97)
\lineto(-2.19,-1.03)
\lineto(-2.11,-1.09)
\lineto(-2.02,-1.15)
\lineto(-1.93,-1.2)
\lineto(-1.87,-1.22)
\lineto(-1.82,-1.25)
\lineto(-1.76,-1.27)
\lineto(-1.7,-1.29)
\lineto(-1.64,-1.3)
\lineto(-1.58,-1.32)
\lineto(-1.52,-1.33)
\lineto(-1.46,-1.35)
\lineto(-1.39,-1.36)
\lineto(-1.33,-1.36)
\lineto(-1.27,-1.37)
\lineto(-1.2,-1.38)
\lineto(-1.14,-1.38)
\lineto(-1.07,-1.38)
\lineto(-1.01,-1.38)
\lineto(-0.99,-1.38)
\lineto(-0.99,-1.38)
\lineto(-0.98,-1.38)
\lineto(-0.98,-1.38)
\lineto(-0.98,-1.38)
\lineto(-0.98,-1.38)
\lineto(-0.98,-1.38)
\lineto(-0.98,-1.38)
\moveto(-2.42,-0.37)
\lineto(-2.42,-0.37)
\lineto(-2.42,-0.37)
\lineto(-2.42,-0.37)
\lineto(-2.42,-0.37)
\lineto(-2.42,-0.37)
\lineto(-2.42,-0.37)
\lineto(-2.42,-0.37)
\lineto(-2.42,-0.37)
\lineto(-2.42,-0.37)
\lineto(-2.42,-0.37)
\lineto(-2.42,-0.37)
\lineto(-2.42,-0.37)
\lineto(-2.42,-0.37)
\lineto(-2.42,-0.37)
\lineto(-2.42,-0.38)
\lineto(-2.42,-0.38)
\lineto(-2.43,-0.39)
\lineto(-2.43,-0.4)
\lineto(-2.44,-0.44)
\lineto(-2.45,-0.51)
\lineto(-2.45,-0.56)
\lineto(-2.45,-0.62)
\lineto(-2.43,-0.68)
\lineto(-2.41,-0.74)
\lineto(-2.38,-0.81)
\lineto(-2.36,-0.85)
}
\psplot[linecolor=qqttzz]{-7.02}{6.76}{(--13.67-3.44*x)/4.92}
\pscustom[linecolor=ccqqqq]{\moveto(-0.03,0.79)
\lineto(0.02,0.82)
\lineto(0.08,0.86)
\lineto(0.14,0.89)
\lineto(0.22,0.92)
\lineto(0.3,0.95)
\lineto(0.38,0.98)
\lineto(0.48,1)
\lineto(0.57,1.01)
\lineto(0.62,1.02)
\lineto(0.68,1.02)
\lineto(0.73,1.03)
\lineto(0.78,1.03)
\lineto(0.83,1.03)
\lineto(0.88,1.02)
\lineto(0.93,1.02)
\lineto(0.99,1.02)
\lineto(1.04,1.01)
\lineto(1.09,1)
\lineto(1.19,0.98)
\lineto(1.28,0.95)
\lineto(1.36,0.92)
\lineto(1.44,0.89)
\lineto(1.51,0.85)
\lineto(1.58,0.81)
\lineto(1.63,0.76)
\lineto(1.68,0.72)
\lineto(1.72,0.68)
\lineto(1.75,0.63)
\lineto(1.78,0.59)
\lineto(1.8,0.54)
\lineto(1.81,0.49)
\lineto(1.82,0.44)
\lineto(1.82,0.39)
\lineto(1.82,0.34)
\lineto(1.81,0.29)
\lineto(1.79,0.24)
\lineto(1.76,0.19)
\lineto(1.73,0.15)
\lineto(1.69,0.1)
\lineto(1.65,0.06)
\lineto(1.61,0.03)
\lineto(1.56,-0.01)
\lineto(1.51,-0.04)
\lineto(1.45,-0.07)
\lineto(1.4,-0.1)
\lineto(1.34,-0.12)
\lineto(1.28,-0.14)
\lineto(1.22,-0.16)
\lineto(1.17,-0.17)
\lineto(1.11,-0.19)
\lineto(1.05,-0.2)
\lineto(1,-0.2)
\lineto(0.95,-0.21)
\lineto(0.9,-0.21)
\lineto(0.85,-0.22)
\lineto(0.8,-0.22)
\lineto(0.75,-0.22)
\lineto(0.71,-0.22)
\lineto(0.67,-0.21)
\lineto(0.63,-0.21)
\lineto(0.59,-0.21)
\lineto(0.55,-0.2)
\lineto(0.52,-0.2)
\lineto(0.49,-0.19)
\lineto(0.45,-0.18)
\lineto(0.42,-0.18)
\lineto(0.4,-0.17)
\lineto(0.37,-0.16)
\lineto(0.34,-0.16)
\lineto(0.32,-0.15)
\lineto(0.29,-0.14)
\lineto(0.27,-0.14)
\lineto(0.25,-0.13)
\lineto(0.23,-0.12)
\lineto(0.21,-0.11)
\lineto(0.19,-0.11)
\lineto(0.18,-0.1)
\lineto(0.16,-0.09)
\lineto(0.15,-0.08)
\lineto(0.13,-0.08)
\lineto(0.12,-0.07)
\lineto(0.1,-0.06)
\lineto(0.09,-0.06)
\lineto(0.08,-0.05)
\lineto(0.07,-0.04)
\lineto(0.06,-0.04)
\lineto(0.04,-0.03)
\lineto(0.03,-0.02)
\lineto(0.02,-0.02)
\lineto(0.02,-0.01)
\lineto(0.01,0)
\lineto(0,0)
\lineto(0,0)
\lineto(0,0)
\lineto(0,0)
\lineto(0,0)
\lineto(0,0)
\lineto(0,0)
\lineto(0,0)
\lineto(0,0)
\lineto(0,0)
\lineto(0,0)
\lineto(0,0)
\lineto(0,0)
\lineto(0,0)
\lineto(0,0)
\lineto(0,0)
\lineto(0,0)
\lineto(0,0)
\lineto(0,0)
\lineto(0,0)
\lineto(0,0)
\lineto(0,0)
\lineto(0,0)
\lineto(0,0)
\lineto(0,0)
\lineto(0,0)
\lineto(0,0)
\lineto(0,0)
\lineto(0,0)
\lineto(0,0)
\lineto(0,0)
\lineto(0,0)
\lineto(-0.01,0)
\lineto(-0.01,0.01)
\lineto(-0.02,0.01)
\lineto(-0.03,0.02)
\lineto(-0.04,0.03)
\lineto(-0.05,0.04)
\lineto(-0.06,0.04)
\lineto(-0.07,0.05)
\lineto(-0.08,0.06)
\lineto(-0.09,0.07)
\lineto(-0.1,0.08)
\lineto(-0.1,0.09)
\lineto(-0.11,0.1)
\lineto(-0.12,0.11)
\lineto(-0.13,0.12)
\lineto(-0.15,0.13)
\lineto(-0.16,0.15)
\lineto(-0.16,0.16)
\lineto(-0.17,0.17)
\lineto(-0.18,0.19)
\lineto(-0.19,0.2)
\lineto(-0.2,0.22)
\lineto(-0.21,0.24)
\lineto(-0.22,0.26)
\lineto(-0.23,0.28)
\lineto(-0.23,0.3)
\lineto(-0.24,0.32)
\lineto(-0.24,0.34)
\lineto(-0.25,0.36)
\lineto(-0.25,0.39)
\lineto(-0.25,0.42)
\lineto(-0.25,0.45)
\lineto(-0.24,0.48)
\lineto(-0.24,0.51)
\lineto(-0.23,0.54)
\lineto(-0.21,0.57)
\lineto(-0.19,0.61)
\lineto(-0.17,0.64)
\lineto(-0.14,0.68)
\lineto(-0.11,0.72)
\lineto(-0.07,0.75)
\lineto(-0.03,0.79)
\lineto(-0.03,0.79)
}
\begin{scriptsize}
\psdots[dotstyle=*,linecolor=blue](0,0)
\rput[bl](-0.22,-0.32){\blue{$O$}}
\end{scriptsize}
\end{pspicture*}
	\end{center}
	\caption{Inversion in an Ellipse of a system of parallel lines.}
	\label{fig:RectasParaleInvElipse}
\end{figure}

\subsection{Elliptic Inversion of Ellipses}

\begin{definition}
If two ellipses  $E_1$ and $E_2$ have parallel axes and have equal eccentricities, then they are said to be of the same semi-form. If in addition the princpal axes are parallel, then they are called  homothetic and it is denoted by $E_1 \sim E_2$.
\end{definition}

\begin{theorem}\label{elipsem}
Let $\chi$ and $\chi'$ be an ellipse and its elliptic inverse curve with respect to the ellipse $E$. Let $\chi$ and $E$ be homothetic curves ($\chi \sim E$), then
\begin{enumerate}[i.]
	\item If $\chi$ not passing through the center of inversion, then  $\chi'$ is an ellipse not passing through the center of inversion  and $\chi' \sim E$, see Figure \ref{fig:elipsesinversioncaso3}.
	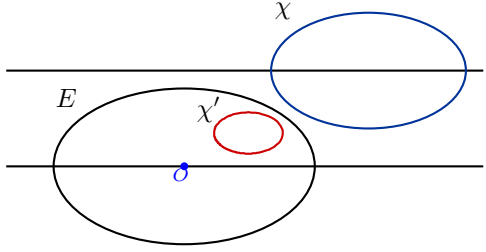
\begin{figure}[h]
	\begin{center}
	\newrgbcolor{qqttzz}{0 0.2 0.6}
\newrgbcolor{ccqqqq}{0.8 0 0}
\psset{xunit=0.7cm,yunit=0.7cm,algebraic=true,dotstyle=o,dotsize=3pt 0,linewidth=0.8pt,arrowsize=3pt 2,arrowinset=0.25}
\begin{pspicture*}(-3.38,-1.7)(5.68,3.42)
\rput{0}(0,0){\psellipse(0,0)(2.5,1.5)}
\rput[tl](-2.44,1.48){$E$}
\psplot{-3.38}{5.68}{(-0-0*x)/4}
\psplot{-3.38}{5.68}{(--7.28-0*x)/4}
\rput{0}(3.5,1.82){\psellipse[linecolor=qqttzz](0,0)(1.87,1.12)}
\pscustom[linecolor=ccqqqq]{\moveto(0.57,0.69)
\lineto(0.58,0.7)
\lineto(0.58,0.72)
\lineto(0.59,0.73)
\lineto(0.59,0.75)
\lineto(0.6,0.76)
\lineto(0.61,0.78)
\lineto(0.63,0.8)
\lineto(0.64,0.81)
\lineto(0.66,0.83)
\lineto(0.67,0.85)
\lineto(0.7,0.87)
\lineto(0.72,0.89)
\lineto(0.75,0.9)
\lineto(0.78,0.92)
\lineto(0.81,0.94)
\lineto(0.85,0.96)
\lineto(0.85,0.96)
\lineto(0.86,0.96)
\lineto(0.86,0.96)
\lineto(0.86,0.96)
\lineto(0.86,0.96)
\lineto(0.86,0.96)
\lineto(0.86,0.96)
\lineto(0.86,0.96)
\lineto(0.86,0.96)
\lineto(0.86,0.96)
\lineto(0.86,0.96)
\lineto(0.86,0.96)
\lineto(0.86,0.96)
\lineto(0.86,0.96)
\lineto(0.86,0.96)
\lineto(0.86,0.96)
\lineto(0.86,0.96)
\lineto(0.86,0.96)
\lineto(0.86,0.96)
\lineto(0.86,0.96)
\lineto(0.86,0.96)
\lineto(0.86,0.96)
\lineto(0.86,0.96)
\lineto(0.86,0.96)
\lineto(0.86,0.96)
\lineto(0.86,0.96)
\lineto(0.86,0.96)
\lineto(0.86,0.96)
\lineto(0.86,0.96)
\lineto(0.86,0.96)
\lineto(0.86,0.96)
\lineto(0.86,0.96)
\lineto(0.86,0.96)
\lineto(0.87,0.96)
\lineto(0.88,0.97)
\lineto(0.89,0.97)
\lineto(0.93,0.99)
\lineto(0.98,1)
\lineto(1.04,1.01)
\lineto(1.09,1.02)
\lineto(1.15,1.02)
\lineto(1.22,1.03)
\lineto(1.28,1.02)
\lineto(1.35,1.02)
\lineto(1.42,1.01)
\lineto(1.48,0.99)
\lineto(1.55,0.97)
\lineto(1.61,0.95)
\lineto(1.67,0.92)
\lineto(1.72,0.89)
\lineto(1.77,0.85)
\lineto(1.8,0.81)
\lineto(1.83,0.77)
\lineto(1.85,0.73)
\lineto(1.87,0.69)
\lineto(1.87,0.64)
\lineto(1.87,0.6)
\lineto(1.86,0.56)
\lineto(1.85,0.53)
\lineto(1.83,0.49)
\lineto(1.81,0.46)
\lineto(1.78,0.43)
\lineto(1.75,0.41)
\lineto(1.72,0.38)
\lineto(1.69,0.36)
\lineto(1.65,0.34)
\lineto(1.62,0.32)
\lineto(1.58,0.31)
\lineto(1.55,0.3)
\lineto(1.52,0.29)
\lineto(1.49,0.28)
\lineto(1.45,0.27)
\lineto(1.42,0.26)
\lineto(1.39,0.26)
\lineto(1.36,0.25)
\lineto(1.33,0.25)
\lineto(1.3,0.25)
\lineto(1.28,0.24)
\lineto(1.25,0.24)
\lineto(1.23,0.24)
\lineto(1.2,0.24)
\lineto(1.18,0.24)
\lineto(1.16,0.24)
\lineto(1.14,0.25)
\lineto(1.11,0.25)
\lineto(1.09,0.25)
\lineto(1.08,0.25)
\lineto(1.06,0.25)
\lineto(1.04,0.26)
\lineto(1.02,0.26)
\lineto(1,0.26)
\lineto(0.99,0.27)
\lineto(0.97,0.27)
\lineto(0.96,0.28)
\lineto(0.94,0.28)
\lineto(0.93,0.28)
\lineto(0.91,0.29)
\lineto(0.9,0.29)
\lineto(0.89,0.3)
\lineto(0.87,0.3)
\lineto(0.86,0.31)
\lineto(0.85,0.31)
\lineto(0.84,0.32)
\lineto(0.82,0.32)
\lineto(0.81,0.33)
\lineto(0.8,0.33)
\lineto(0.79,0.34)
\lineto(0.78,0.34)
\lineto(0.77,0.35)
\lineto(0.76,0.35)
\lineto(0.75,0.36)
\lineto(0.74,0.37)
\lineto(0.73,0.37)
\lineto(0.72,0.38)
\lineto(0.72,0.38)
\lineto(0.71,0.39)
\lineto(0.7,0.4)
\lineto(0.69,0.4)
\lineto(0.68,0.41)
\lineto(0.68,0.42)
\lineto(0.67,0.42)
\lineto(0.66,0.43)
\lineto(0.65,0.44)
\lineto(0.65,0.45)
\lineto(0.64,0.45)
\lineto(0.63,0.46)
\lineto(0.63,0.47)
\lineto(0.62,0.48)
\lineto(0.61,0.49)
\lineto(0.61,0.5)
\lineto(0.6,0.5)
\lineto(0.6,0.51)
\lineto(0.59,0.52)
\lineto(0.59,0.53)
\lineto(0.59,0.54)
\lineto(0.58,0.55)
\lineto(0.58,0.56)
\lineto(0.57,0.57)
\lineto(0.57,0.58)
\lineto(0.57,0.59)
\lineto(0.57,0.61)
\lineto(0.57,0.62)
\lineto(0.57,0.63)
\lineto(0.57,0.64)
\lineto(0.57,0.66)
\lineto(0.57,0.67)
\lineto(0.57,0.68)
\lineto(0.57,0.69)
}
\rput[tl](1.72,3.12){$\chi$}
\rput[tl](0.24,1.32){$\chi'$}
\begin{scriptsize}
\psdots[dotstyle=*,linecolor=blue](0,0)
\rput[bl](-0.22,-0.28){\blue{$O$}}
\end{scriptsize}
\end{pspicture*}
	\end{center}
	\caption{Theorem \ref{elipsem}, Case $i$.}
	\label{fig:elipsesinversioncaso3}
\end{figure}

	\item If $\chi$ passing through the center of inversion, then  $\chi'$  is a line, see Figure \ref{fig:elipsesinversioncaso4y5}.
	\begin{figure}[h]
	\begin{center}
		\newrgbcolor{qqttzz}{0 0.2 0.6}
\newrgbcolor{ccqqqq}{0.8 0 0}
\psset{xunit=0.7cm,yunit=0.7cm,algebraic=true,dotstyle=o,dotsize=3pt 0,linewidth=0.8pt,arrowsize=3pt 2,arrowinset=0.25}
\begin{pspicture*}(-4.68,-1.7)(4.74,2.44)
\rput{0}(0,0){\psellipse(0,0)(2.5,1.5)}
\rput[tl](-2.28,1.62){$E$}
\psplot[linestyle=dashed,dash=1pt 1pt]{-4.68}{4.74}{(-0-0*x)/4}
\psplot[linestyle=dashed,dash=1pt 1pt]{-4.68}{4.74}{(--2.72-0*x)/4}
\rput[tl](-3.86,1.86){$\chi$}
\rput[tl](0.82,2.46){$\chi'$}
\rput{0}(-1.69,0.68){\psellipse[linecolor=qqttzz](0,0)(2.04,1.22)}
\pscustom[linecolor=ccqqqq]{\moveto(-0.95,0.8)
\lineto(-0.98,0.78)
\lineto(-1,0.76)
\lineto(-1.02,0.74)
\lineto(-1.05,0.72)
\lineto(-1.07,0.7)
\lineto(-1.09,0.68)
\lineto(-1.11,0.66)
\lineto(-1.13,0.64)
\lineto(-1.16,0.62)
\lineto(-1.18,0.6)
\lineto(-1.2,0.58)
\lineto(-1.22,0.56)
\lineto(-1.24,0.55)
\lineto(-1.26,0.53)
\lineto(-1.28,0.51)
\lineto(-1.3,0.49)
\lineto(-1.32,0.47)
\lineto(-1.34,0.45)
\lineto(-1.37,0.43)
\lineto(-1.39,0.41)
\lineto(-1.41,0.39)
\lineto(-1.43,0.37)
\lineto(-1.45,0.36)
\lineto(-1.47,0.34)
\lineto(-1.5,0.32)
\lineto(-1.52,0.3)
\lineto(-1.53,0.29)
\lineto(-1.54,0.28)
\lineto(-1.54,0.28)
\lineto(-1.54,0.28)
\lineto(-1.54,0.28)
\lineto(-1.54,0.28)
\lineto(-1.54,0.28)
\lineto(-1.54,0.28)
\lineto(-1.54,0.28)
\lineto(-1.54,0.28)
\lineto(-1.54,0.28)
\lineto(-1.54,0.28)
\lineto(-1.54,0.28)
\lineto(-1.54,0.28)
\lineto(-1.54,0.28)
\lineto(-1.54,0.28)
\lineto(-1.54,0.28)
\lineto(-1.54,0.28)
\lineto(-1.54,0.28)
\lineto(-1.54,0.28)
\lineto(-1.54,0.28)
\lineto(-1.54,0.28)
\lineto(-1.54,0.28)
\lineto(-1.54,0.28)
\lineto(-1.54,0.28)
\lineto(-1.54,0.28)
\lineto(-1.54,0.28)
\lineto(-1.54,0.28)
\lineto(-1.54,0.28)
\lineto(-1.54,0.28)
\lineto(-1.54,0.28)
\lineto(-1.54,0.28)
\lineto(-1.55,0.27)
\lineto(-1.56,0.26)
\lineto(-1.58,0.24)
\lineto(-1.6,0.22)
\lineto(-1.62,0.2)
\lineto(-1.65,0.18)
\lineto(-1.67,0.16)
\lineto(-1.7,0.13)
\lineto(-1.73,0.11)
\lineto(-1.75,0.09)
\lineto(-1.78,0.06)
\lineto(-1.81,0.04)
\lineto(-1.84,0.01)
\lineto(-1.87,-0.02)
\lineto(-1.9,-0.05)
\lineto(-1.93,-0.08)
\lineto(-1.97,-0.11)
\lineto(-2,-0.14)
\lineto(-2.04,-0.17)
\lineto(-2.08,-0.21)
\lineto(-2.12,-0.24)
\lineto(-2.16,-0.28)
\lineto(-2.21,-0.32)
\lineto(-2.26,-0.36)
\lineto(-2.31,-0.41)
\lineto(-2.36,-0.45)
\lineto(-2.37,-0.47)
\lineto(-2.37,-0.47)
\lineto(-2.37,-0.47)
\lineto(-2.37,-0.47)
\lineto(-2.37,-0.47)
\lineto(-2.37,-0.47)
\lineto(-2.37,-0.47)
\lineto(-2.37,-0.47)
\lineto(-2.37,-0.47)
\lineto(-2.37,-0.47)
\moveto(-0.18,1.5)
\lineto(-0.18,1.5)
\lineto(-0.18,1.5)
\lineto(-0.18,1.5)
\lineto(-0.18,1.5)
\lineto(-0.18,1.5)
\lineto(-0.18,1.5)
\lineto(-0.18,1.5)
\lineto(-0.18,1.5)
\lineto(-0.18,1.5)
\lineto(-0.18,1.5)
\lineto(-0.18,1.5)
\lineto(-0.18,1.5)
\lineto(-0.18,1.5)
\lineto(-0.18,1.5)
\lineto(-0.18,1.5)
\lineto(-0.18,1.5)
\lineto(-0.18,1.49)
\lineto(-0.18,1.49)
\lineto(-0.18,1.49)
\lineto(-0.19,1.49)
\lineto(-0.2,1.48)
\lineto(-0.22,1.45)
\lineto(-0.27,1.42)
\lineto(-0.32,1.37)
\lineto(-0.36,1.33)
\lineto(-0.41,1.29)
\lineto(-0.45,1.25)
\lineto(-0.49,1.22)
\lineto(-0.53,1.18)
\lineto(-0.56,1.15)
\lineto(-0.6,1.12)
\lineto(-0.63,1.09)
\lineto(-0.66,1.06)
\lineto(-0.69,1.03)
\lineto(-0.73,1.01)
\lineto(-0.75,0.98)
\lineto(-0.78,0.96)
\lineto(-0.81,0.93)
\lineto(-0.84,0.91)
\lineto(-0.86,0.88)
\lineto(-0.89,0.86)
\lineto(-0.91,0.84)
\lineto(-0.94,0.82)
\lineto(-0.95,0.8)
}
\pscustom[linecolor=ccqqqq]{\moveto(-2.41,-0.5)
\lineto(-2.41,-0.5)
\lineto(-2.48,-0.56)
\lineto(-2.54,-0.62)
\lineto(-2.61,-0.68)
\lineto(-2.68,-0.74)
\lineto(-2.77,-0.82)
\lineto(-2.85,-0.9)
\lineto(-2.95,-0.98)
\lineto(-3,-1.03)
\lineto(-3.06,-1.08)
\lineto(-3.12,-1.13)
\lineto(-3.18,-1.18)
\lineto(-3.24,-1.24)
\lineto(-3.31,-1.3)
\lineto(-3.38,-1.37)
\lineto(-3.46,-1.43)
\lineto(-3.54,-1.51)
\lineto(-3.63,-1.58)
\lineto(-3.72,-1.67)
\lineto(-3.82,-1.75)
\moveto(0.83,2.39)
\lineto(0.76,2.33)
\lineto(0.69,2.27)
\lineto(0.62,2.21)
\lineto(0.56,2.16)
\lineto(0.51,2.11)
\lineto(0.45,2.06)
\lineto(0.4,2.01)
\lineto(0.35,1.97)
\lineto(0.26,1.88)
\lineto(0.17,1.81)
\lineto(0.09,1.74)
\lineto(0.02,1.67)
\lineto(-0.04,1.61)
\lineto(-0.11,1.56)
\lineto(-0.17,1.51)
\lineto(-0.17,1.5)
\lineto(-0.18,1.5)
\lineto(-0.18,1.5)
\lineto(-0.18,1.5)
\lineto(-0.18,1.5)
\lineto(-0.18,1.5)
\lineto(-0.18,1.5)
\lineto(-0.18,1.5)
\lineto(-0.18,1.5)
\lineto(-0.18,1.5)
\moveto(-2.37,-0.47)
\lineto(-2.37,-0.47)
\lineto(-2.37,-0.47)
\lineto(-2.37,-0.47)
\lineto(-2.37,-0.47)
\lineto(-2.37,-0.47)
\lineto(-2.37,-0.47)
\lineto(-2.37,-0.47)
\lineto(-2.37,-0.47)
\lineto(-2.37,-0.47)
\lineto(-2.37,-0.47)
\lineto(-2.37,-0.47)
\lineto(-2.37,-0.47)
\lineto(-2.38,-0.47)
\lineto(-2.38,-0.47)
\lineto(-2.38,-0.47)
\lineto(-2.38,-0.47)
\lineto(-2.38,-0.47)
\lineto(-2.38,-0.47)
\lineto(-2.38,-0.47)
\lineto(-2.39,-0.48)
\lineto(-2.4,-0.49)
\lineto(-2.41,-0.5)
\lineto(-2.41,-0.5)
}
\begin{scriptsize}
\psdots[dotstyle=*,linecolor=blue](0,0)
\rput[bl](-0.22,-0.28){\blue{$O$}}
\end{scriptsize}
\end{pspicture*}
	\end{center}
	\caption{Theorem \ref{elipsem}, Case $ii$. }
	\label{fig:elipsesinversioncaso4y5}
\end{figure}
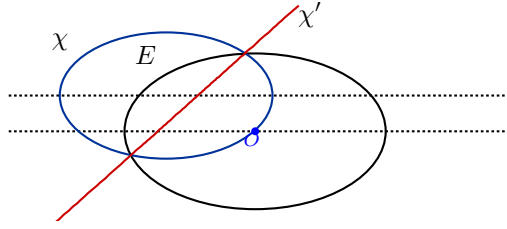

	\item If $\chi$ is orthogonal to $E$, then $\chi'$ is the ellipse itself.

\end{enumerate}
\end{theorem}
\begin{proof}
\emph{i.} Let $\chi$ be the ellipse $\frac{x^2}{a^2} + \frac{y^2}{b^2} + Dx + Ey + F=0 \ (F\neq 0)$. Applying $\psi$ to this equation gives $\frac{x^2}{a^2} + \frac{y^2}{b^2} + \frac{D}{F}x + \frac{E}{F}y + \frac{1}{F}=0$. Indeed
\begin{align*}
\frac{x^2}{a^2} + \frac{y^2}{b^2} + Dx + Ey + F=&0\\
\frac{\left(\frac{a^2b^2x}{b^2x^2+a^2y^2}\right)^2}{a^2} + \frac{\left(\frac{a^2b^2y}{b^2x^2+a^2y^2}\right)^2}{b^2} + D\left(\frac{a^2b^2x}{b^2x^2+a^2y^2}\right) + E\left(\frac{a^2b^2y}{b^2x^2+a^2y^2}\right) + F=&0\\
a^2b^4x^2 + a^4b^2y^2 + Da^2b^2 x (b^2x^2 + a^2y^2) + Ea^2b^2 x (b^2x^2 + a^2y^2) + F(b^2x^2 + a^2y^2)^2=&0\\
\frac{x^2}{a^2} + \frac{y^2}{b^2} + Dx\left(\frac{x^2}{a^2} + \frac{y^2}{b^2}\right) + Ey\left(\frac{x^2}{a^2} + \frac{y^2}{b^2}\right) + F\left(\frac{x^2}{a^2} + \frac{y^2}{b^2}\right)^2 =&0\\
\left(\frac{x^2}{a^2} + \frac{y^2}{b^2}\right)\left(1 +  Dx + Ey + F\left(\frac{x^2}{a^2} + \frac{y^2}{b^2}\right)\right) =&0\\
\frac{x^2}{a^2} + \frac{y^2}{b^2} + \frac{D}{F}x + \frac{E}{F}y + \frac{1}{F}=&0
\end{align*}
\emph{ii} and \emph{iii}  proof run like in $i$.
% Let $\chi$ be the ellipse $\frac{x^2}{a^2} + \frac{y^2}{b^2} + Dx + Ey=0$ ($D$ and $E$ not both zero). Applying $\psi$ to this equation and simplifying  gives the line  $Dx + Ey+1=0$, then the result is clear.
\end{proof}

\subsection{Elliptic Inversion of Other Curves}

\begin{theorem}\label{curves} The inverse of any conic not of the same semi-form as the central conic of inversion and passing through the center of inversion is a cubic curve.
\end{theorem}
\begin{proof}
Let $\chi$ be the conic $Ax^2 + Bxy + Cy^2 + Dx + Ey =0,  (A=1/a^2, B=0$ and $C=1/b^2$ cannot hold simultaneously). Applying $\psi$  to this equation, we  have
\begin{align*}
Aa^4b^4x^2 +  Ba^4b^4xy + Ca^4b^4y^2 + Db^2x^3 + Da^2xy^2 + Eb^2x^2y + Ea^2y^3&=0
\end{align*}
\end{proof}

\begin{theorem} The inverse of any conic not of the same semi-form as the central conic of inversion and not passing through the center of inversion is a curve of the fourth degree.
\end{theorem}
\begin{proof}
Similar to Theorem \ref{curves}.
\end{proof}
\begin{example}
In Figure \ref{fig:CircunInvElipse}, we show the elliptic inverse  of a circumference $\mathcal{\chi}$.
\end{example}

\begin{figure}[h]
	\begin{center}
	\newrgbcolor{qqttzz}{0 0.2 0.6}
\newrgbcolor{ccqqqq}{0.8 0 0}
\psset{xunit=0.7cm,yunit=0.7cm,algebraic=true,dotstyle=o,dotsize=3pt 0,linewidth=0.8pt,arrowsize=3pt 2,arrowinset=0.25}
\begin{pspicture*}(-4.92,-1.68)(3.02,3.92)
\rput{0}(0,0){\psellipse(0,0)(2.5,1.5)}
\rput[tl](2.14,1.5){$E$}
\pscircle[linecolor=qqttzz](-2.8,1.96){1.2}
\pscustom[linecolor=ccqqqq]{\moveto(-1.19,1.3)
\lineto(-1.19,1.3)
\lineto(-1.13,1.29)
\lineto(-1.07,1.27)
\lineto(-1.02,1.26)
\lineto(-0.97,1.24)
\lineto(-0.87,1.21)
\lineto(-0.79,1.17)
\lineto(-0.72,1.13)
\lineto(-0.66,1.09)
\lineto(-0.61,1.06)
\lineto(-0.56,1.02)
\lineto(-0.52,0.99)
\lineto(-0.49,0.95)
\lineto(-0.46,0.92)
\lineto(-0.44,0.89)
\lineto(-0.42,0.87)
\lineto(-0.4,0.84)
\lineto(-0.38,0.81)
\lineto(-0.37,0.79)
\lineto(-0.36,0.77)
\lineto(-0.35,0.75)
\lineto(-0.34,0.73)
\lineto(-0.34,0.71)
\lineto(-0.33,0.69)
\lineto(-0.33,0.68)
\lineto(-0.33,0.66)
\lineto(-0.33,0.65)
\lineto(-0.33,0.63)
\lineto(-0.33,0.62)
\lineto(-0.33,0.61)
\lineto(-0.33,0.6)
\lineto(-0.33,0.59)
\lineto(-0.33,0.58)
\lineto(-0.34,0.57)
\lineto(-0.34,0.56)
\lineto(-0.35,0.55)
\lineto(-0.35,0.54)
\lineto(-0.36,0.54)
\lineto(-0.36,0.53)
\lineto(-0.37,0.52)
\lineto(-0.37,0.52)
\lineto(-0.38,0.51)
\lineto(-0.39,0.51)
\lineto(-0.4,0.5)
\lineto(-0.41,0.5)
\lineto(-0.41,0.49)
\lineto(-0.42,0.49)
\lineto(-0.43,0.48)
\lineto(-0.44,0.48)
\lineto(-0.45,0.48)
\lineto(-0.47,0.47)
\lineto(-0.48,0.47)
\lineto(-0.49,0.47)
\lineto(-0.5,0.46)
\lineto(-0.52,0.46)
\lineto(-0.53,0.46)
\lineto(-0.54,0.46)
\lineto(-0.56,0.45)
\lineto(-0.58,0.45)
\lineto(-0.59,0.45)
\lineto(-0.61,0.44)
\lineto(-0.63,0.44)
\lineto(-0.64,0.44)
\lineto(-0.66,0.44)
\lineto(-0.68,0.43)
\lineto(-0.7,0.43)
\lineto(-0.72,0.43)
\lineto(-0.75,0.42)
\lineto(-0.77,0.42)
\lineto(-0.79,0.42)
\lineto(-0.81,0.41)
\lineto(-0.84,0.41)
\lineto(-0.86,0.4)
\lineto(-0.89,0.4)
\lineto(-0.9,0.4)
\lineto(-0.91,0.4)
\lineto(-0.91,0.4)
\lineto(-0.91,0.4)
\lineto(-0.91,0.4)
\lineto(-0.91,0.4)
\lineto(-0.91,0.4)
\lineto(-0.91,0.4)
\lineto(-0.91,0.4)
\lineto(-0.91,0.4)
\lineto(-0.91,0.4)
\lineto(-0.91,0.4)
\lineto(-0.91,0.4)
\lineto(-0.91,0.4)
\lineto(-0.91,0.4)
\lineto(-0.91,0.4)
\lineto(-0.91,0.4)
\lineto(-0.91,0.4)
\lineto(-0.91,0.4)
\lineto(-0.91,0.4)
\lineto(-0.91,0.4)
\lineto(-0.91,0.4)
\lineto(-0.91,0.4)
\lineto(-0.91,0.4)
\lineto(-0.91,0.4)
\lineto(-0.91,0.4)
\lineto(-0.91,0.4)
\lineto(-0.91,0.4)
\lineto(-0.91,0.4)
\lineto(-0.91,0.39)
\lineto(-0.91,0.39)
\lineto(-0.92,0.39)
\lineto(-0.93,0.39)
\lineto(-0.96,0.39)
\lineto(-0.98,0.38)
\lineto(-1.01,0.37)
\lineto(-1.04,0.37)
\lineto(-1.07,0.36)
\lineto(-1.1,0.35)
\lineto(-1.14,0.34)
\lineto(-1.17,0.34)
\lineto(-1.2,0.33)
\lineto(-1.23,0.32)
\lineto(-1.27,0.31)
\lineto(-1.3,0.3)
\lineto(-1.34,0.29)
\lineto(-1.37,0.28)
\lineto(-1.41,0.27)
\lineto(-1.45,0.26)
\lineto(-1.49,0.25)
\lineto(-1.52,0.24)
\lineto(-1.56,0.23)
\lineto(-1.61,0.22)
\lineto(-1.65,0.21)
\lineto(-1.69,0.2)
\lineto(-1.74,0.2)
\lineto(-1.78,0.19)
\lineto(-1.83,0.18)
\lineto(-1.88,0.18)
\lineto(-1.93,0.18)
\lineto(-1.99,0.18)
\lineto(-2.05,0.18)
\lineto(-2.1,0.19)
\lineto(-2.16,0.19)
\lineto(-2.23,0.21)
\lineto(-2.29,0.22)
\lineto(-2.36,0.25)
\lineto(-2.42,0.27)
\lineto(-2.44,0.28)
\lineto(-2.45,0.29)
\lineto(-2.45,0.29)
\lineto(-2.45,0.29)
\lineto(-2.45,0.29)
\lineto(-2.45,0.29)
\lineto(-2.45,0.29)
\lineto(-2.45,0.29)
\lineto(-2.45,0.29)
\lineto(-2.45,0.29)
\moveto(-1.22,1.31)
\lineto(-1.22,1.31)
\lineto(-1.22,1.31)
\lineto(-1.22,1.31)
\lineto(-1.22,1.31)
\lineto(-1.22,1.31)
\lineto(-1.22,1.31)
\lineto(-1.22,1.31)
\lineto(-1.22,1.31)
\lineto(-1.22,1.31)
\lineto(-1.22,1.31)
\lineto(-1.22,1.31)
\lineto(-1.22,1.31)
\lineto(-1.22,1.31)
\lineto(-1.22,1.31)
\lineto(-1.22,1.31)
\lineto(-1.22,1.31)
\lineto(-1.22,1.31)
\lineto(-1.21,1.31)
\lineto(-1.21,1.31)
\lineto(-1.19,1.3)
\lineto(-1.19,1.3)
}
\pscustom[linecolor=ccqqqq]{\moveto(-2.6,1.01)
\lineto(-2.5,1.09)
\lineto(-2.45,1.13)
\lineto(-2.38,1.17)
\lineto(-2.32,1.2)
\lineto(-2.24,1.23)
\lineto(-2.17,1.26)
\lineto(-2.09,1.28)
\lineto(-2.01,1.3)
\lineto(-1.93,1.32)
\lineto(-1.85,1.33)
\lineto(-1.77,1.34)
\lineto(-1.69,1.34)
\lineto(-1.61,1.34)
\lineto(-1.53,1.34)
\lineto(-1.46,1.34)
\lineto(-1.39,1.33)
\lineto(-1.32,1.32)
\lineto(-1.25,1.31)
\lineto(-1.23,1.31)
\lineto(-1.23,1.31)
\lineto(-1.22,1.31)
\lineto(-1.22,1.31)
\lineto(-1.22,1.31)
\lineto(-1.22,1.31)
\lineto(-1.22,1.31)
\lineto(-1.22,1.31)
\lineto(-1.22,1.31)
\moveto(-2.45,0.29)
\lineto(-2.45,0.29)
\lineto(-2.45,0.29)
\lineto(-2.45,0.29)
\lineto(-2.45,0.29)
\lineto(-2.45,0.29)
\lineto(-2.45,0.29)
\lineto(-2.45,0.29)
\lineto(-2.45,0.29)
\lineto(-2.45,0.29)
\lineto(-2.45,0.29)
\lineto(-2.45,0.29)
\lineto(-2.45,0.29)
\lineto(-2.45,0.29)
\lineto(-2.45,0.29)
\lineto(-2.45,0.29)
\lineto(-2.45,0.29)
\lineto(-2.46,0.29)
\lineto(-2.46,0.29)
\lineto(-2.46,0.29)
\lineto(-2.47,0.3)
\lineto(-2.48,0.3)
\lineto(-2.51,0.32)
\lineto(-2.56,0.36)
\lineto(-2.62,0.41)
\lineto(-2.67,0.47)
\lineto(-2.71,0.54)
\lineto(-2.74,0.61)
\lineto(-2.75,0.7)
\lineto(-2.73,0.79)
\lineto(-2.7,0.88)
\lineto(-2.64,0.97)
\lineto(-2.6,1.01)
}
\rput[tl](-4.78,3.26){$\chi$}
\rput[tl](-2.94,1.58){$\chi'$}
\begin{scriptsize}
\psdots[dotstyle=*,linecolor=blue](0,0)
\rput[bl](-0.22,-0.35){\blue{$O$}}
\psdots[dotstyle=*,linecolor=blue](-2.8,1.96)
\end{scriptsize}
\end{pspicture*}
	\end{center}
	\caption{Inversion in an Ellipse  of a  Circumference.}
	\label{fig:CircunInvElipse}
\end{figure}
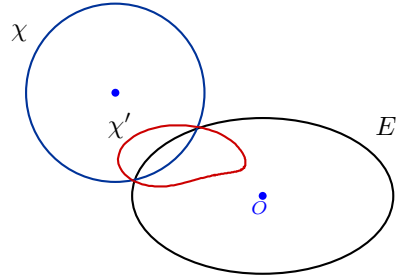

\begin{example}
In Figure \ref{fig:ParabolasInvEliptica}, we show the elliptic inverse   of a parabola $\chi$.
\end{example}

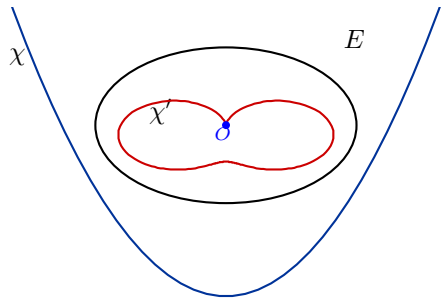
\begin{figure}[h]
\begin{center}
\newrgbcolor{qqttzz}{0 0.2 0.6}
\newrgbcolor{ccqqqq}{0.8 0 0}
\psset{xunit=0.7cm,yunit=0.7cm,algebraic=true,dotstyle=o,dotsize=3pt 0,linewidth=0.8pt,arrowsize=3pt 2,arrowinset=0.25}
\begin{pspicture*}(-4.42,-3.52)(4.6,2.24)
\rput{0}(0,0){\psellipse(0,0)(2.5,1.5)}
\rput[tl](2.24,1.82){$E$}
\rput[tl](-4.12,1.46){$\chi$}
\rput[tl](-1.46,0.5){$\chi'$}
\psplot{-4.42}{4.6}{(-4-0*x)/1}
\rput{0}(0,-3.25){\psplot[linecolor=qqttzz]{-6}{6}{x^2/2/1.5}}
\pscustom[linecolor=ccqqqq]{\moveto(-0.12,-0.71)
\lineto(-0.13,-0.71)
\lineto(-0.14,-0.71)
\lineto(-0.15,-0.71)
\lineto(-0.17,-0.72)
\lineto(-0.18,-0.72)
\lineto(-0.2,-0.72)
\lineto(-0.21,-0.73)
\lineto(-0.23,-0.73)
\lineto(-0.26,-0.74)
\lineto(-0.28,-0.75)
\lineto(-0.31,-0.75)
\lineto(-0.34,-0.76)
\lineto(-0.37,-0.77)
\lineto(-0.41,-0.78)
\lineto(-0.46,-0.79)
\lineto(-0.52,-0.8)
\lineto(-0.58,-0.82)
\lineto(-0.67,-0.83)
\lineto(-0.71,-0.83)
\lineto(-0.77,-0.84)
\lineto(-0.83,-0.84)
\lineto(-0.89,-0.84)
\lineto(-0.97,-0.84)
\lineto(-1.05,-0.84)
\lineto(-1.14,-0.83)
\lineto(-1.19,-0.82)
\lineto(-1.24,-0.82)
\lineto(-1.3,-0.81)
\lineto(-1.35,-0.79)
\lineto(-1.41,-0.78)
\lineto(-1.47,-0.76)
\lineto(-1.53,-0.74)
\lineto(-1.6,-0.71)
\lineto(-1.66,-0.68)
\lineto(-1.73,-0.64)
\lineto(-1.79,-0.6)
\lineto(-1.85,-0.56)
\lineto(-1.9,-0.51)
\lineto(-1.95,-0.45)
\lineto(-1.99,-0.39)
\lineto(-2.02,-0.32)
\lineto(-2.04,-0.25)
\lineto(-2.04,-0.17)
\lineto(-2.04,-0.1)
\lineto(-2.01,-0.02)
\lineto(-1.97,0.05)
\lineto(-1.92,0.12)
\lineto(-1.85,0.19)
\lineto(-1.78,0.25)
\lineto(-1.69,0.3)
\lineto(-1.6,0.34)
\lineto(-1.5,0.38)
\lineto(-1.4,0.41)
\lineto(-1.3,0.43)
\lineto(-1.21,0.45)
\lineto(-1.11,0.46)
\lineto(-1.02,0.47)
\lineto(-0.94,0.47)
\lineto(-0.86,0.46)
\lineto(-0.78,0.46)
\lineto(-0.71,0.45)
\lineto(-0.65,0.43)
\lineto(-0.59,0.42)
\lineto(-0.54,0.41)
\lineto(-0.49,0.39)
\lineto(-0.4,0.36)
\lineto(-0.33,0.33)
\lineto(-0.27,0.3)
\lineto(-0.19,0.25)
\lineto(-0.13,0.2)
\lineto(-0.09,0.16)
\lineto(-0.06,0.12)
\lineto(-0.04,0.1)
\lineto(-0.02,0.07)
\lineto(-0.01,0.05)
\lineto(-0.01,0.04)
\lineto(0,0.03)
\lineto(0,0.02)
\lineto(0,0.01)
\lineto(0,0)
\lineto(0,0)
\lineto(0,0)
\lineto(0,0)
\lineto(0,0)
\lineto(0,0)
\lineto(0,0)
\lineto(0,0)
\lineto(0,0)
\lineto(0,0)
\lineto(0,0)
\lineto(0,0)
\lineto(0,0)
\lineto(0,0)
\lineto(0,0)
\lineto(0,0)
\lineto(0,0)
\lineto(0,0)
\lineto(0,0)
\lineto(0,0)
\lineto(0,0)
\lineto(0,0)
\lineto(0,0)
\lineto(0,0)
\lineto(0,0)
\lineto(0,0)
\lineto(0,0)
\lineto(0,0)
\lineto(0,0)
\lineto(0,0)
\lineto(0,0)
\lineto(0,0)
\lineto(0,0)
\lineto(0,0)
\lineto(0,0)
\lineto(0,0)
\lineto(0,0)
\lineto(0,0.01)
\lineto(0,0.02)
\lineto(0,0.03)
\lineto(0.01,0.04)
\lineto(0.01,0.05)
\lineto(0.02,0.07)
\lineto(0.04,0.1)
\lineto(0.06,0.12)
\lineto(0.09,0.16)
\lineto(0.13,0.2)
\lineto(0.19,0.25)
\lineto(0.27,0.3)
\lineto(0.33,0.33)
\lineto(0.4,0.36)
\lineto(0.49,0.39)
\lineto(0.54,0.41)
\lineto(0.59,0.42)
\lineto(0.65,0.43)
\lineto(0.71,0.45)
\lineto(0.78,0.46)
\lineto(0.86,0.46)
\lineto(0.94,0.47)
\lineto(1.02,0.47)
\lineto(1.11,0.46)
\lineto(1.21,0.45)
\lineto(1.3,0.43)
\lineto(1.4,0.41)
\lineto(1.5,0.38)
\lineto(1.6,0.34)
\lineto(1.69,0.3)
\lineto(1.78,0.25)
\lineto(1.85,0.19)
\lineto(1.92,0.12)
\lineto(1.97,0.05)
\lineto(2.01,-0.02)
\lineto(2.04,-0.1)
\lineto(2.04,-0.17)
\lineto(2.04,-0.25)
\lineto(2.02,-0.32)
\lineto(1.99,-0.39)
\lineto(1.95,-0.45)
\lineto(1.9,-0.51)
\lineto(1.85,-0.56)
\lineto(1.79,-0.6)
\lineto(1.73,-0.64)
\lineto(1.66,-0.68)
\lineto(1.6,-0.71)
\lineto(1.53,-0.74)
\lineto(1.47,-0.76)
\lineto(1.41,-0.78)
\lineto(1.35,-0.79)
\lineto(1.3,-0.81)
\lineto(1.24,-0.82)
\lineto(1.19,-0.82)
\lineto(1.14,-0.83)
\lineto(1.05,-0.84)
\lineto(0.97,-0.84)
\lineto(0.89,-0.84)
\lineto(0.83,-0.84)
\lineto(0.77,-0.84)
\lineto(0.71,-0.83)
\lineto(0.67,-0.83)
\lineto(0.58,-0.82)
\lineto(0.52,-0.8)
\lineto(0.46,-0.79)
\lineto(0.41,-0.78)
\lineto(0.37,-0.77)
\lineto(0.34,-0.76)
\lineto(0.31,-0.75)
\lineto(0.28,-0.75)
\lineto(0.26,-0.74)
\lineto(0.23,-0.73)
\lineto(0.21,-0.73)
\lineto(0.2,-0.72)
\lineto(0.18,-0.72)
\lineto(0.17,-0.72)
\lineto(0.15,-0.71)
\lineto(0.14,-0.71)
\lineto(0.13,-0.71)
\lineto(0.12,-0.71)
\lineto(0.11,-0.7)
\lineto(0.1,-0.7)
\lineto(0.09,-0.7)
\lineto(0.08,-0.7)
\lineto(0.07,-0.7)
\lineto(0.06,-0.7)
\lineto(0.06,-0.7)
\lineto(0.05,-0.69)
\lineto(0.04,-0.69)
\lineto(0.04,-0.69)
\lineto(0.03,-0.69)
\lineto(0.03,-0.69)
\lineto(0.02,-0.69)
\lineto(0.01,-0.69)
\lineto(0.01,-0.69)
\lineto(0,-0.69)
\lineto(0,-0.69)
\lineto(-0.01,-0.69)
\lineto(-0.01,-0.69)
\lineto(-0.02,-0.69)
\lineto(-0.02,-0.69)
\lineto(-0.03,-0.69)
\lineto(-0.04,-0.69)
\lineto(-0.04,-0.69)
\lineto(-0.05,-0.69)
\lineto(-0.06,-0.7)
\lineto(-0.06,-0.7)
\lineto(-0.07,-0.7)
\lineto(-0.08,-0.7)
\lineto(-0.09,-0.7)
\lineto(-0.1,-0.7)
\lineto(-0.1,-0.7)
\lineto(-0.11,-0.7)
\lineto(-0.12,-0.71)
}
\begin{scriptsize}
\psdots[dotstyle=*,linecolor=blue](0,0)
\rput[bl](-0.22,-0.28){\blue{$O$}}
\psdots[dotstyle=*,linecolor=blue](0,-4)
\rput[bl](0.08,-3.88){\blue{$B$}}
\end{scriptsize}
\end{pspicture*}
	\end{center}
	\caption{Inversion in an Ellipse of a Parabola.}
	\label{fig:ParabolasInvEliptica}
\end{figure}

\begin{example}
In Figure  \ref{fig:hiperbolainversioncaso1}, we show the elliptic inverse  of an hyperbola  $\chi$.
\end{example}

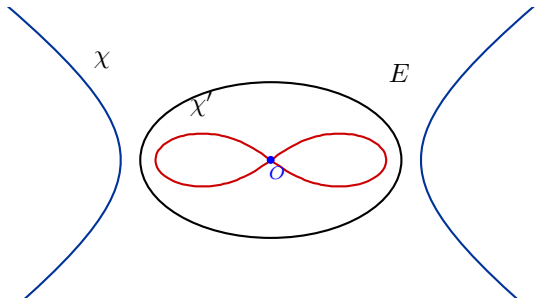
\begin{figure}[h]
	\begin{center}
		\newrgbcolor{qqttzz}{0 0.2 0.6}
\newrgbcolor{ccqqqq}{0.8 0 0}
\psset{xunit=0.7cm,yunit=0.7cm,algebraic=true,dotstyle=o,dotsize=3pt 0,linewidth=0.8pt,arrowsize=3pt 2,arrowinset=0.25}
\begin{pspicture*}(-6.48,-2.62)(6.86,2.9)
\rput{0}(0,0){\psellipse(0,0)(2.5,1.5)}
\rput[tl](2.24,1.82){$E$}
\rput[tl](-3.36,2.06){$\chi$}
\rput[tl](-1.54,1.3){$\chi'$}
\rput{0}(0,0){\parametricplot[linecolor=qqttzz]{-0.99}{0.99}{2.85*(1+t^2)/(1-t^2)|2.03*2*t/(1-t^2)}}
\rput{0}(0,0){\parametricplot[linecolor=qqttzz]{-0.99}{0.99}{2.85*(-1-t^2)/(1-t^2)|2.03*(-2)*t/(1-t^2)}}
\pscustom[linecolor=ccqqqq]{\moveto(-1.27,0.5)
\lineto(-1.2,0.5)
\lineto(-1.13,0.49)
\lineto(-1.07,0.49)
\lineto(-1,0.48)
\lineto(-0.93,0.46)
\lineto(-0.87,0.45)
\lineto(-0.8,0.43)
\lineto(-0.74,0.41)
\lineto(-0.68,0.39)
\lineto(-0.63,0.37)
\lineto(-0.57,0.35)
\lineto(-0.52,0.32)
\lineto(-0.42,0.27)
\lineto(-0.33,0.22)
\lineto(-0.25,0.17)
\lineto(-0.18,0.13)
\lineto(-0.13,0.09)
\lineto(-0.08,0.05)
\lineto(-0.04,0.03)
\lineto(-0.02,0.01)
\lineto(-0.01,0)
\lineto(0,0)
\lineto(0,0)
\lineto(0,0)
\lineto(0,0)
\lineto(0,0)
\lineto(0,0)
\lineto(0,0)
\lineto(0,0)
\lineto(0,0)
\lineto(0,0)
\lineto(0,0)
\lineto(0,0)
\lineto(0,0)
\lineto(0,0)
\lineto(0,0)
\lineto(0,0)
\lineto(0,0)
\lineto(0,0)
\lineto(0.01,-0.01)
\lineto(0.03,-0.02)
\lineto(0.06,-0.05)
\lineto(0.11,-0.08)
\lineto(0.16,-0.11)
\lineto(0.23,-0.16)
\lineto(0.3,-0.21)
\lineto(0.39,-0.26)
\lineto(0.49,-0.31)
\lineto(0.54,-0.33)
\lineto(0.59,-0.35)
\lineto(0.64,-0.38)
\lineto(0.7,-0.4)
\lineto(0.76,-0.42)
\lineto(0.82,-0.44)
\lineto(0.89,-0.45)
\lineto(0.95,-0.47)
\lineto(1.02,-0.48)
\lineto(1.09,-0.49)
\lineto(1.16,-0.5)
\lineto(1.22,-0.5)
\lineto(1.29,-0.5)
\lineto(1.36,-0.5)
\lineto(1.43,-0.5)
\lineto(1.5,-0.49)
\lineto(1.56,-0.48)
\lineto(1.62,-0.47)
\lineto(1.68,-0.46)
\lineto(1.74,-0.44)
\lineto(1.8,-0.42)
\lineto(1.85,-0.4)
\lineto(1.94,-0.35)
\lineto(2.02,-0.3)
\lineto(2.08,-0.25)
\lineto(2.13,-0.19)
\lineto(2.16,-0.14)
\lineto(2.18,-0.08)
\lineto(2.19,-0.03)
\lineto(2.19,0.02)
\lineto(2.18,0.07)
\lineto(2.17,0.12)
\lineto(2.14,0.18)
\lineto(2.09,0.23)
\lineto(2.04,0.29)
\lineto(1.96,0.34)
\lineto(1.88,0.39)
\lineto(1.83,0.41)
\lineto(1.78,0.43)
\lineto(1.72,0.45)
\lineto(1.66,0.46)
\lineto(1.6,0.48)
\lineto(1.54,0.49)
\lineto(1.47,0.49)
\lineto(1.4,0.5)
\lineto(1.34,0.5)
\lineto(1.27,0.5)
\lineto(1.2,0.5)
\lineto(1.13,0.49)
\lineto(1.06,0.49)
\lineto(0.99,0.48)
\lineto(0.93,0.46)
\lineto(0.86,0.45)
\lineto(0.8,0.43)
\lineto(0.74,0.41)
\lineto(0.68,0.39)
\lineto(0.62,0.37)
\lineto(0.57,0.35)
\lineto(0.52,0.32)
\lineto(0.42,0.27)
\lineto(0.33,0.22)
\lineto(0.25,0.17)
\lineto(0.18,0.13)
\lineto(0.12,0.09)
\lineto(0.08,0.05)
\lineto(0.04,0.03)
\lineto(0.02,0.01)
\lineto(0.01,0)
\lineto(0,0)
\lineto(0,0)
\lineto(0,0)
\lineto(0,0)
\lineto(0,0)
\lineto(0,0)
\lineto(0,0)
\lineto(0,0)
\lineto(0,0)
\lineto(-0.01,-0.01)
\lineto(-0.03,-0.02)
\lineto(-0.06,-0.05)
\lineto(-0.11,-0.08)
\lineto(-0.16,-0.11)
\lineto(-0.23,-0.16)
\lineto(-0.3,-0.21)
\lineto(-0.39,-0.26)
\lineto(-0.49,-0.31)
\lineto(-0.54,-0.33)
\lineto(-0.59,-0.35)
\lineto(-0.64,-0.38)
\lineto(-0.7,-0.4)
\lineto(-0.76,-0.42)
\lineto(-0.82,-0.44)
\lineto(-0.89,-0.45)
\lineto(-0.95,-0.47)
\lineto(-1.02,-0.48)
\lineto(-1.09,-0.49)
\lineto(-1.16,-0.5)
\lineto(-1.22,-0.5)
\lineto(-1.29,-0.5)
\lineto(-1.36,-0.5)
\lineto(-1.43,-0.5)
\lineto(-1.5,-0.49)
\lineto(-1.56,-0.48)
\lineto(-1.62,-0.47)
\lineto(-1.68,-0.46)
\lineto(-1.74,-0.44)
\lineto(-1.8,-0.42)
\lineto(-1.85,-0.4)
\lineto(-1.94,-0.35)
\lineto(-2.02,-0.3)
\lineto(-2.08,-0.25)
\lineto(-2.13,-0.19)
\lineto(-2.16,-0.14)
\lineto(-2.18,-0.08)
\lineto(-2.19,-0.03)
\lineto(-2.19,0.02)
\lineto(-2.18,0.07)
\lineto(-2.17,0.12)
\lineto(-2.14,0.18)
\lineto(-2.09,0.23)
\lineto(-2.04,0.29)
\lineto(-1.96,0.34)
\lineto(-1.88,0.39)
\lineto(-1.83,0.41)
\lineto(-1.78,0.43)
\lineto(-1.72,0.45)
\lineto(-1.66,0.46)
\lineto(-1.6,0.48)
\lineto(-1.54,0.49)
\lineto(-1.47,0.49)
\lineto(-1.4,0.5)
\lineto(-1.34,0.5)
\lineto(-1.27,0.5)
}
\begin{scriptsize}
\psdots[dotstyle=*,linecolor=blue](0,0)
\rput[bl](-0.04,-0.36){\blue{$O$}}
\end{scriptsize}
\end{pspicture*}
	\end{center}
	\caption{Inversion in an Ellipse of a Hyperbola.}
	\label{fig:hiperbolainversioncaso1}
\end{figure}

Note that the inversion in an ellipse is not  conformal.
\section{Pappus Elliptic Chain}
The classical inversion  has a lot of applications, such as the Pappus Chain Theorem, Feuerbach's Theorem, Steiner Porism, the problem of Apollonius, among others \cite{BLA, OGI, PED}. In this section, we generalize The Pappus Chain Theorem  with respect to ellipses.

\begin{theorem}
Let $E$ be a semiellipse with principal diameter $\overline{AB}$, and $E'$ and $E_0$  semiellipses on the same  side of $\overline{AB}$ with principal diameters $\overline{AC}$ and $\overline{CD}$  respectively, and   $E\sim E_0,  E_0 \sim E'$, see Figure \ref{cadenaeli}. Let $E_1, E_2, \dots$ be a sequence of ellipses tangent to $E$ and $E'$, such that $E_n$ is tangent to $E_{n-1}$ and $E_n \sim E_{n-1}$ for all $n\geq 1$. Let $r_n$ be the semi-minor axis of $E_n$ and $h_n$ the distance of the center of $E_n$ from $\overline{AB}$. Then $h_{n}=2nr_{n}$
\end{theorem}
\begin{proof}
Let $\psi_i$ the elliptic inversion  such that  $\psi(E_{i})=E_{i}$, (in Figure \ref{cadenaeli} we select  $i=2$), i.e., $\psi_i=\mathcal{E}(B,t_i)$, where $t_i$ is the length of the tangent segment to the Ellipse $E$ from the point $B$.

\begin{figure}[h]
    \begin{center}
        \includegraphics[scale=1]{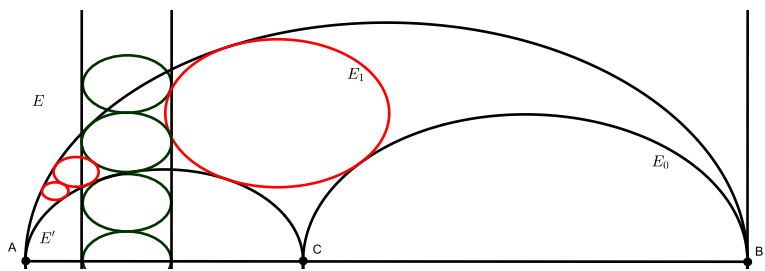}
    \end{center}
    \caption{Elliptic Pappus Chain.}
    \label{cadenaeli}
\end{figure}

By Theorem \ref{elipsem}, $\psi_i(E)$ and $\psi_i(E_0)$ are perpendicular lines to the line $\stackrel{\longleftrightarrow}{AB}$  and  tangentes to the ellipse $E_{i}$. Hence, ellipses $\psi_i(E_{1}), \psi_i(E_{2}), \dots$ will also invert to tangent ellipses  to parallel lines $\psi_i(E)$ and $\psi_i(E_0)$. Whence  $h_{i}=2ir_{i}$.
\end{proof}

\section{Concluding remarks}
The study of elliptic inversion  suggests interesting and challenging problems. For
example, generalized the Steiner Porism or Apollonius Problems with respect to ellipses.

\end{document}